\documentclass[reqno]{amsart}
\usepackage{color}
\usepackage{amsfonts,amsmath,amsthm,amsfonts,amssymb,amstext}
\usepackage[dvips]{graphicx}
\usepackage{psfrag}
\usepackage{url}
\usepackage[dvipsnames]{xcolor}

\makeatletter \@addtoreset{equation}{section} \makeatother

\renewcommand\thetable{\thesection.\@arabic\c@table}

\theoremstyle{plain}

\newtheorem{theorem}{Theorem}[section]
\newtheorem{proposition}{Proposition}[section]
\newtheorem{lemma}{Lemma}[section]
\newtheorem{corollary}{Corollary}[section]

\newtheorem{remark}{Remark}[section]

\begin{document}

\title[Dimension Estimates and Continuity of Pressure]{Dimension Estimates for Non-conformal Repellers and Continuity of Sub-additive Topological Pressure}

\author{Yongluo Cao}
\address{Departament of Mathematics, East China Normal University\\
 Shanghai 200062, P.R. China}
\address{Departament of Mathematics, Soochow University\\
Suzhou 215006, Jiangsu, P.R. China}
\address{Center for Dynamical Systems and Differential Equation, Soochow University\\
Suzhou 215006, Jiangsu, P.R. China}
\email{ylcao@suda.edu.cn}

\author{Yakov Pesin}
\address{Department of Mathematics \\ Pennsylvania State University \\ University Park, PA 16802, USA}
\email{pesin@math.psu.edu}

\author{Yun Zhao}
\address{School of Mathematical Sciences, Soochow University\\
Suzhou 215006, Jiangsu, P.R. China}
\address{Center for Dynamical Systems and Differential Equation, Soochow University\\
Suzhou 215006, Jiangsu, P.R. China
}
\email{zhaoyun@suda.edu.cn}

\thanks{The first author is partially supported by NSFC (11771317, 11790274), the second author is partially supported  by NSF grant DMS-1400027, the third author is  partially supported by NSFC (11871361,11790274).}

\date{\today}

\begin{abstract}Given a non-conformal repeller $\Lambda$ of a $C^{1+\gamma}$ map, we study the Hausdorff dimension of the repeller and continuity of the sub-additive topological pressure for the sub-additive singular valued potentials. Such a potential always possesses an equilibrium state. We then use a substantially modified version of Katok's approximating argument, to construct a compact invariant set on which the corresponding dynamical quantities (such as Lyapunov exponents and metric entropy) are close to that of the equilibrium measure. This allows us to establish  continuity of the sub-additive topological pressure and obtain a sharp lower bound of the Hausdorff dimension of the repeller. The latter is given by the zero of the super-additive topological pressure.
\end{abstract}



\maketitle


\section{Introduction}

Dimension is an important characteristic of invariant sets and measures of dynamical systems, (see the books \cite{Ba08,Ba11,fal03,Pes97,pu} where the role of dimension in the theory of dynamical systems is well explained). In this regard, different versions of dimension have been put forward to characterize various dynamical phenomena. However, computing these fractal invariants is usually a challenging problem, because they depend on the microscopic structure of the set.

For repellers of conformal expanding maps Bowen \cite{Bo79} and Ruelle \cite{Ru82} found that their Hausdorff dimension is a solution of an equation involving topological pressure. More precisely, if $\Lambda$ is an isolated compact invariant set of a $C^{1+\gamma}$ conformal expanding map $f$ and $f|_\Lambda$ is topologically mixing, then the Hausdorff dimension of $\Lambda$ is given by the unique root $s$ of the following equation known as \emph{Bowen's equation}:
\begin{eqnarray}\label{bowen-equ}
P(f|_{\Lambda}, -s\log \|D_xf\|)=0,
\end{eqnarray}
where $P(f|_\Lambda,\cdot)$ denotes the topological pressure. Moreover, the equilibrium measure corresponding to the potential function $-s\log \|D_xf\|$ has dimension $s$ and hence, is the measure of maximal dimension.

In this conformal setting, other dimension characteristics such as lower and upper box dimensions are equal to $s$ and the topological pressure $P(f|_{\Lambda}, -\log \|D_xf\|)$ is continuous with respect to $f$ in the $C^1$ topology, hence, so is the Hausdorff dimension. Furthermore, one can also establish analytic dependence of the Hausdorff dimension on the maps with some specific classes of invariant sets called \emph{dynamically defined sets}, see Ruelle's work on hyperbolic Julia sets \cite{Ru82}, which in turn was inspired by Bowen's work on the limits sets for quasi-Fuchsian groups \cite{Bo79}. In the case of a basic set $\Lambda$ for an Axiom A diffeomorphism $f$ on a compact surface $M$, the dependence of the Hausdorff dimension under small perturbations of $f$ was shown to be continuous by McCluskey and Manning \cite{MM83} and to be $C^\infty$ by Ma\~{n}\'{e} \cite{ma} (see also \cite{jl}). Further, Pollicott \cite{P15} showed the analytic dependence of the Hausdorff dimension of the basic set for real analytic Smale horseshoe maps. See also \cite{KKPW,VW,WO} for related works.

The case of non-conformal repellers, which we consider in this paper, is drastically different. In this case the Hausdorff dimension of the repeller may not be equal to its (lower or upper) box dimension and there may not be any \emph{invariant} measure of maximal dimension (see \cite{DS}). The study of dimension in this case is a substantially more complicated problem and to approach it different notions of topological pressure have been introduced, which allow one to obtain some upper bounds on the dimension. In \cite{ba96} Barreira proved that formula \eqref{bowen-equ} holds in a variety of settings. In \cite{fal94}, Falconer defined the topological pressure for sub-additive potentials and obtained the variational principle under the \emph{bounded distortion} or \emph{$1$-bunched condition}. He also proved that the zero of the topological pressure of sub-additive singular valued potentials gives an upper bound of the Hausdorff dimension of repellers. In \cite{zhang}, Zhang introduced a new version of Bowen's equation which involves the limit of a sequence of topological pressures for singular valued potentials, and proved that the unique solution of this equation is an upper bound of the Hausdorff dimension of repellers. Finally, in \cite{bch}, using thermodynamic formalism for sub-additive potentials developed in \cite{cfh}, the authors showed that the zero of the topological pressure of sub-additive singular valued potentials $\Phi_f(t)=\{-\varphi^t(\cdot,f^n)\}_{n\ge 1}$ (see precise definition in Section \ref{svp}) gives an upper bound of the Hausdorff dimension of repellers, and furthermore, that the upper bounds obtained in the previous works \cite{bch,fal94,zhang} are all equal. We refer the reader to \cite{cp} and \cite{bg} for a detailed description of the recent progress in dimension theory of dynamical systems.

We also note that in the case of iterated function system, the zero of topological pressure of sub-additive singular valued potentials always gives an upper bound for the dimension of the set, and sometimes gives the exact value of the dimension, see \cite{F88} and \cite{S88} for details. Recently, Feng and Shmerkin \cite{fs14} have proved that the sub-additive topological pressure of singular valued potentials is continuous in the class of generic self-affine maps. Consequently, the dimension of ``typical'' self-affine sets is also continuous. This resolves a folklore open problem in fractal geometry.

The motivation for this paper is two-fold. First, in view of Zhang's result \cite{zhang}, it is interesting to know whether there is a lower bound for the Hausdorff dimension of a non-conformal repeller that can be obtained as the zero of topological pressure. A natural way to proceed is to replace the sub-additive singular valued potentials with super-additive ones (see Section \ref{svp} the definition). However, in doing so one faces a difficult problem: it is not known whether the corresponding super-additive topological pressure defined in a usual way via separated sets satisfies the variational principle for a general topological dynamical system, although it is true for some special systems, see \cite{bch} for details. To overcome this difficulty, we define the super-additive topological pressure via variational principle (see Section 2.4). This allows us to obtain a lower bound on the Hausdorff dimension of repellers for a general $C^{1+\gamma}$ expanding map, see Theorem \ref{dim-main}, which to the best of our knowledge, is the sharpest lower bound currently known.

Second, in view of Feng and Shmerkin's result \cite{fs14}, it is natural to ask whether the topological pressure of sub-additive singular valued potentials is continuous with respect to the dynamics $f$ in an appropriate topology. One of our main result shows that the map
$f\mapsto P(f|_\Lambda,\Phi_f(t))$ is continuous in the $C^1$ topology within the class of $C^{1+\gamma}$ expanding maps $f$, see Theorem \ref{con-top}. Further, we show that the unique root $t(f,\Lambda)$ of Bowen's equation $P(f|_\Lambda,\Phi_f(t))=0$ is exactly the \emph{Carath\'{e}odory singular dimension} of the repeller (see Section \ref{cont-cs-dim} for details). Thus replacing the Hausdorff dimension of a repeller in the non-conformal setting with its appropriately chosen Carath\'{e}odory singular dimension, one obtains a \emph{precise} value of the dimension. Furthermore, our result on continuity of topological pressure thus implies that the Carath\'{e}odory singular dimension also varies continuously with the dynamics, see Theorem \ref{carath}.

Our main innovation in obtaining a lower bound on the dimension and in establishing continuity of topological pressure, which distinguishes our approach from previous ones, is based on some powerful new results in non-uniform hyperbolicity theory for non-invertible maps. Therefore, we require that $f$ is of class of smoothness $C^{1+\gamma}$. These new results are given by Theorems 5.1 and 5.2 and are of independent interest in hyperbolicity theory. The first one is in the spirit of Katok's approximation argument, see \cite{kat} or \cite[Supplement S.5]{K95} but we need a substantially stronger version of it. Namely, we show that in the setting of expanding maps, given an invariant ergodic measure
$\mu$, there is a compact invariant set $K$ with dominated splitting whose expansion rates are close to the Lyapunov exponents of $\mu$; moreover, the topological entropy of $f|K$ is close to the metric entropy of $\mu$ (see Theorem 5.1).

Further, a similar result holds for any sufficiently small perturbation $g$ of $f$: there is a compact set $K$ invariant under $g$ with dominated splitting whose expansion rates are close to the Lyapunov exponents of $\mu$; moreover, the topological entropy of $g|K$ is close to the metric entropy of $\mu$ (see Theorem 5.2). We shall apply these theorems in the proofs of our two main results: 1) a lower bound for the Hausdorff dimension of the repeller (see Theorem \ref{dim-main} and Section \ref{main-proof}) where we choose $\mu$ such that the corresponding \emph{free energy} $h_\mu(f)+\mathcal{F}_*(\Psi_f(s),\mu)$ (here $\mathcal{F}_*(\Psi_f(s),\mu)$ is the potential associated with the sequence
$\Psi_f(s)$ of super-additive singular valued potentials, see Sections 2.4 and 3.1) is sufficiently close to the super-additive topological pressure; and 2) continuity of sub-additive topological pressure (see Theorem 3.4 and Section \ref{s6.6}) where we choose $\mu$ to be an equilibrium measure for the singular valued sub-additive potential $\Phi_f(t)$.

Results similar to our Theorem 5.1 were obtained in some particular situations by 1) Misiurewicz and Szlenk \cite{ms} for continuous and piecewise monotone maps of the interval; 2) Przytycki and Urba\'{n}ski \cite{pu} for holomorphic maps in the case of a measure with only positive Lyapunov exponent; 3) Persson and Schmeling \cite{ps} for dyadic Diophantine approximations. For $C^{1+\gamma}$ maps results related to Katok's approximation construction (and hence, in some way to our Theorem 5.1) were obtained by Chung \cite{chu}, Yang \cite{yang}, Gelfert \cite{G1, G2}. See also \cite{M85,M88,M89,sa02,sa03,ls, morsch} that represent works close to this topic.

For general $C^{1+\gamma}$ diffeomorphisms (i.e., invertible maps) Theorem 5.1 was shown by Avila, Crovisier, and Wilkinson in \cite{acw}. While their construction of compact invariant set is based on the shadowing lemma, our approach is more geometrical and gives the desired compact set via a Cantor-like construction. The advantage of our approach in the settings of expanding maps is that it allows us to treat the case of non-invertible maps and also obtain a similar result for small $C^1$ perturbations thus proving Theorem 5.2.

The paper is organized as follows. In Section \ref{np}, we recall various notions of topological pressure and dimension. In particularly, we define the super-additive topological pressure via variational relation and discuss  some properties of super-additive topological pressure. In Section \ref{dimension} we state our main results on dimension estimates and continuity of the topological pressure and dimension. In particular, we introduce the super- and sub- additive singular valued potentials and the sub-additive potential as well as the corresponding pressure functions. We then show that given a (non-conformal) repeller, the zero of the topological pressure function of the super-additive singular valued potential gives the lower bound on its Hausdorff dimension and the zero of the topological pressure function of a sub-additive potential gives the upper bound on its upper box dimension (under the dominated splitting condition). In Section \ref{cont-cs-dim}, we introduce the notion of Carath\'{e}odory singular dimension of a repeller and prove that it depends continuously on the map. In Section \ref{ly-exp}, we prove our two main technical results discussed above on construction of an invariant compact subset with dominated splitting and large topological entropy for the map $f$ and its small perturbations. Section \ref{main-thms} contains the proofs of our main results. In the last section \ref{application} we discuss continuity of topological pressure for the case of matrix cocycles over subshifts of finite type.
\vskip0.03in
\noindent{\bf Acknowledgements.} The authors would like to thank the referees for their careful reading and valuable comments which helped improve the paper substantially. The authors would like to thank Professor Dejun Feng and Wen Huang for their suggestions and comments. Ya. P. wants to thank Mittag-Leffler Institute and the Department of Mathematics of Weizmann Institute of Science where part of this work was done for their hospitality.


\section{Notations and Preliminaries}\label{np}

\subsection{Dimensions of sets and measures}We recall some notions and basic facts from dimension theory, see the books \cite{fal03} and \cite{Pes97} for detailed introduction.

Let $X$ be a compact metric space equipped with a metric $d$. Given  a subset $Z$ of $X$, for $s\geq 0$ and $\delta>0$, define
\[
\mathcal{H}_{\delta}^{s}(Z):=\inf \left\{\sum_{i}|U_i|^s: \
Z\subset \bigcup_{i}U_i,~|U_i|\leq \delta\right\}
\]
where $|\cdot|$ denotes the diameter of a set. The quantity
$\mathcal{H}^{s}(Z):=\lim\limits_{\delta\rightarrow 0}\mathcal{H}_{\delta}^{s}(Z)$ is called the {\em $s-$dimensional Hausdorff measure} of $Z$. Define the {\em Hausdorff dimension} of $Z$, denoted by $\dim_H  Z$, as follows:
\[
\dim_H  Z =\inf \{s:\ \mathcal{H}^{s}(Z)=0\}=\sup \{s: \mathcal{H}^{s}(Z)=\infty\}.
\]
Further define the \emph{lower and upper box dimensions} of $Z$ respectively by
$$
\underline{\dim}_B Z=\liminf\limits_{\delta\to0}\frac{\log N(Z,\delta)}{-\log\delta}\ \text{and}\ \overline{\dim}_B Z=\limsup\limits_{\delta\to0}\frac{\log N(Z,\delta)}{-\log\delta},
$$
where $N(Z,\delta)$ denotes the smallest number of balls of radius $\delta$ needed to cover the set $Z$. Clearly, $\dim_H  Z \le \underline{\dim}_B Z\le \overline{\dim}_B Z$ for each subset $Z\subset X$.

If $\mu$ is a probability measure on $X$, then the {\em Hausdorff dimension and the lower and upper box dimension of $\mu$} are defined respectively by
\begin{eqnarray*}
\dim_H \mu&=&\inf \Big\{\dim_HY: Y\subseteq X, \mu(Y)=1\Big\},\\
\underline{\dim}_B\mu&=&\lim\limits_{\delta\to0}\inf\Big\{\underline{\dim}_BY:
Y\subseteq X, \mu(Y)\geq1-\delta\Big\},\\
\overline{\dim}_B\mu&=&\lim\limits_{\delta\to0}\inf\Big\{\overline{\dim}_BY:
Y\subseteq X, \mu(Y)\geq1-\delta\Big\}.
\end{eqnarray*}
The following inequalities are immediate
$\underline{\dim}_H\mu\leq \underline{\dim}_B\mu\leq \overline{\dim}_\mathrm{B}\mu$.

Finally, we define the {\em lower and upper pointwise dimensions of the measure
$\mu$} at the point $x\in X$ by
\begin{equation}\label{def2.9}
\underline{d}_\mu(x)=\liminf\limits_{r\to0}\frac{\log\mu(B(x,r))}{\log r}\ \text{and}\ \overline{d}_\mu(x)=\limsup\limits_{r\to0}\frac{\log\mu(B(x,r))}{\log r},
\end{equation}
where $B(x,r)=\{y\in X: d(x,y)<r\}$.

In particular, if there exists a number $s$ such that
$$
\lim\limits_{r\to0}\frac{\log\mu(B(x, r))}{\log r} = s
$$
for $\mu$-almost every $x\in X$, then $\dim_H\mu = s$, see \cite{young}.

\subsection{Repellers for expanding maps} Let $f:M\to M$ be a smooth map of a $m_0$-dimensional compact smooth Riemannian manifold $M$ and let $\Lambda:=\Lambda_f$ be a compact $f$-invariant subset of $M$. We call $\Lambda$ a \emph{repeller} for $f$ and $f$ \emph{expanding} if
\begin{enumerate}
\item[(1)]there exists an open neighborhood $U$ of $\Lambda$ such that
$\Lambda=\{x\in U: f^n(x)\in U,\, \text{ for all }\, n\ge 0\}$;
\item[(2)]there is $\kappa > 1 $ such that
$$
\|D_xf v \| \ge \kappa \|v\|, \, \text{ for all } x \in \Lambda,\, \text{ and } v \in T_xM,
$$
where $\|\cdot\|$ is the norm induced by the Riemannian metric on $M$.
\end{enumerate}

\subsection{Markov partitions and Gibbs measures}Let $\Lambda$ be a repeller of a $C^1$ expanding map $f$. Assume that $f|\Lambda$ is topologically transitive. Then there exists a partition $\{P_1,\dots, P_k\}$ of $\Lambda$ which satisfies
\begin{enumerate}
\item $P_i\ne\emptyset$ and $\overline{\text{int }(P_i)}=P_i$;
\item $\text{int }(P_i)\cap\text{int }(P_j)=\emptyset$ if $i\ne j$;
\item for each $i$ the set $f(P_i)$ is the union of some of the sets $P_j$ from the partition.
\end{enumerate}
Here $int(\cdot)$ denotes the interior of a set relative to $\Lambda$. Such a partition is called \emph{Markov} (see \cite[Theorem 3.5.2]{pu} for proofs). A repeller admits a Markov partition into subsets of arbitrary small diameter, and in particular, we may assume that the restriction of $f$ on each $P_i$  is injective.

Given a Markov partition $\{P_1,\dots, P_k\}$, a sequence
$\mathbf{i}=(i_0 i_1\dots i_{n-1})$ where $1\le i_j\le k$ is called {\it admissible}, if
$f(P_{i_j})\supset P_{i_{j+1}}$ for $j=0,1,\dots,n-2$; and we write $|\mathbf{i}|=n$ for the length of the sequence $\mathbf{i}$. If $\mathbf{i}=(i_0 i_1\dots i_{n-1})$ is an admissible sequence, we define the cylinder
\begin{equation}\label{cyn}
P_{i_0i_1\dots i_{n-1}}=\bigcap_{j=0}^{n-1} f^{-j}(P_{i_j}).
\end{equation}
Denote by $S_n$ the collection of all admissible sequences of length $n$.

We call a (not necessarily invariant) Borel probability measure $\mu$ on $\Lambda$ a \emph{Gibbs measure} for a continuous function $\varphi$ on $\Lambda$ if there exists $C>0$ such that for every $n\in\mathbb{N}$, every admissible sequence
$(i_0i_1\dots i_{n-1})$, and $x\in P_{i_0i_1\dots i_{n-1}}$ we have
\[
C^{-1}\le\frac{\mu(P_{i_0i_1\dots i_{n-1}})}{\exp[-nP+S_n\varphi(x)]}\le C,
\]
where $P$ is a constant and
$S_n\varphi(x)=\sum_{j=0}^{n-1}\varphi(f^jx)$.

\subsection{Sub-additive and super-additive topological pressures}\label{sub-sup-pressure} See \cite{Ba06} and \cite{fh10,fh16} for more details. Consider a continuous transformation $f: X\to X$ of a compact metric space $X$ equipped with metric $d$.
Denote by $\mathcal{M}(X,f)$ and $\mathcal{M}^e(X,f)$ the set of all $f$-invariant and respectively, ergodic Borel probability measures on $X$. A sequence of continuous functions (potentials) $\Phi=\{\varphi_n\}_{n\ge 1}$ is called \emph{sub-additive}, if
$$
\varphi_{m+n}\le\varphi_n+\varphi_m\circ f^n,~~ \text{ for all } m,n\ge 1.
$$
Similarly, we call a sequence of continuous functions (potentials) $\Psi=\{\psi_n\}_{n\ge 1}$ \emph{super-additive} if $-\Psi=\{-\psi_n\}_{n\ge 1}$ is sub-additive.

For $x, y\in X$ and $n\ge 0$ define the $d_n$-metric on $X$ by
$$
d_n(x,y)=\max\{d(f^i(x),f^i(y)): 0\le i<n\}.
$$
Given $\varepsilon>0$ and $n\ge 0$, denote by
$B_n(x,\varepsilon)=\{y\in X: d_n(x,y)<\varepsilon\}$ Bowen's ball centered at $x$ of radius
$\varepsilon$ and length $n$ and we call a subset $E\subset X$
$(n,\varepsilon)$-separated if $d_n(x,y)>\varepsilon$ for any two distinct points $x,y\in E$.

Given a sub-additive sequence of continuous potentials $\Phi=\{\varphi_n\}_{n\ge 1}$, let
$$
P_n(\Phi,\varepsilon)=\sup\Big\{\sum\limits_{x \in E}
e^{\varphi_n(x)}: E \,\,\text{is an } (n,\varepsilon)-\text{separated subset of }X\Big\}.
$$
The quantity
\begin{equation}\label{pressure-sub}
P(f,\Phi)=\lim\limits_{\varepsilon\to 0}\limsup\limits_{n\to\infty}\frac1n\log P_n(\Phi,\varepsilon)
\end{equation}
is called the {\em sub-additive topological pressure} of $\Phi$. One can show (see \cite{cfh}) that it satisfies the following variational principle:
\begin{equation}\label{var-principle}
P(f,\Phi)=\sup\Big\{h_\mu(f)+\mathcal{F}_*(\Phi,\mu):\; \mu\in
\mathcal{M}(X,f),\,\,\mathcal{F}_*(\Phi,\mu)\ne -\infty\Big\},
\end{equation}
where $h_{\mu}(f)$ is the metric entropy of $f$ with respect to $\mu$ and
\begin{equation}\label{poten}
\mathcal{F}_*(\Phi,\mu)=\lim_{n\to\infty}\frac1n\int\varphi_n d\mu.
\end{equation}
Existence of the above limit can be shown by the standard sub-additive argument.

Given a super-additive sequence of continuous potentials $\Psi=\{\psi_n\}_{n\ge 1}$, we define the {\em super-additive topological pressure} of $\Psi$ by
\begin{equation}\label{pressure-supper}
P_{\mathrm{var}}(f,\Psi):=\sup \Big\{h_\mu(f)+\mathcal{F}_*(\Psi,\mu):\; \mu\in
\mathcal M(X,f) \Big\}.
\end{equation}
Note that for any super-additive sequence of continuous potentials and any $f$-invariant measure $\mu$ we have
$$
\mathcal{F}_*(\Psi,\mu)=\lim_{n\to\infty}\frac1n\int\psi_n\,d\mu=\sup\frac1n\int\psi_n\,d\mu.
$$
Similarly to the definition of the sub-additive topological pressure one can define the notion of super-additive topological pressure $P(f,\Psi)$ using $(n,\varepsilon)$-separated sets.

\begin{remark} It is natural to ask whether the variational principle for $P(f,\Psi)$ holds, i.e.,
$$
P(f,\Psi)=\sup\Big\{h_\mu(f)+\mathcal F_*(\Psi,\mu):\; \mu\in\mathcal M(X,f)\Big\}.
$$
\end{remark}
Given a continuous function $\varphi$, set $\varphi_n=\sum_{i=0}^{n-1}\varphi\circ f^i$. The above definitions of sub/super-additive topological pressures recover the standard notion of topological pressure of a single continuous potential (see \cite{wal82} for details), which we denote by $P(f,\varphi)$. The following result establishes some useful relations between the three notions of the pressure.
\begin{proposition}\label{sup-add-aprox}Let $\Psi=\{\psi_n\}_{n\ge 1}$ be a super-additive sequence of continuous potentials on $X$. Then
\[
P_{\mathrm{var}}(f,\Psi)=\lim\limits_{n\to\infty}P(f,\frac{\psi_n}{n})
=\lim_{n\to\infty}\frac1nP(f^n,\psi_n).
\]
\end{proposition}
\begin{remark}The same results hold for sub-additive topological pressure if the entropy map is upper semi-continuous, see \cite{bch}. However, we do not need this condition in the case of super-additive topological pressure defined by variational relation.
\end{remark}


\section{Main Results: bounds on dimensions and continuity of pressure}\label{dimension}

Unless otherwise stated, throughout this section we assume that $f:M\to M$ is a
$C^{1+\gamma}$ expanding map of a $m_0$-dimensional compact smooth Riemannian manifold $M$, and  $\Lambda$ is a repeller of $f$ such that $f|\Lambda$ is topologically transitive. In this section we establish a lower bound for the Hausdorff dimension as well as an upper bound for the upper box dimension of $\Lambda$ by using super-additive and sub-additive topological pressures respectively. We also prove the continuity of the sub-additive topological pressure of sub-additive singular valued potentials.

\subsection{Singular valued potentials}\label{svp} Let $S$ and $S'$ be two real linear  spaces of the same dimension, each endowed with an inner product, and let $L: S\to S'$ be a linear map. The singular values of $L$ are the square roots of the eigenvalues of $L^*L$.

Given $x\in\Lambda$ and $n\ge 1$, consider the differentiable operator $D_xf^n: T_xM\to T_{f^n(x)}M$ and  denote the singular values of $D_xf^n$ in the decreasing order by
\begin{equation}\label{sing-val}
\alpha_1(x,f^n)\ge\alpha_2(x,f^n)\ge\dots\ge\alpha_{m_0}(x,f^n).
\end{equation}
For $s\in [0,m_0]$, set
\begin{equation}\label{sing-val-psi}
\psi^{s}(x,f^n):
=\sum_{i=1}^{[s]}\log\alpha_i(x,f^n)+
(s-[s])\log\alpha_{[s]+1}(x,f^n)
\end{equation}
and
\begin{equation}\label{sing-val-phi}
\varphi^{t}(x,f^n):
=\sum_{i=m_0-[t]+1}^{m_0}\log\alpha_i(x,f^n)
+(t-[t])\log\alpha_{m_0-[t]}(x,f^n)
\end{equation}
for $t\in [0,m_0]$.
Since $f$ is smooth, the functions $x\mapsto\alpha_i(x,f^n)$, $x\mapsto\psi^s(x,f^n)$ and $x\mapsto\varphi^t(x,f^n)$ are continuous. It is easy to see that for all $n,\ell\in\mathbb{N}$
$$
\begin{aligned}
\psi^s(x,f^{n+\ell})&\le\psi^s(x, f^n)+\psi^s(f^n(x),f^\ell), \\
\varphi^t(x,f^{n+\ell})&\ge\varphi^t(x,f^n)+\varphi^t(f^n(x),f^\ell).
\end{aligned}
$$
It follows that the sequences of functions
\begin{equation}\label{sing-val-poten}
\Psi_f(s):=\{-\psi^s(\cdot,f^n)\}_{n\ge 1} \text{ and } \Phi_f(t):=\{-\varphi^t(\cdot,f^n)\}_{n\ge 1}
\end{equation}
are respectively, super-additive and sub-additive. We call them respectively \emph{super-additive and sub-additive singular valued potentials}.

We consider the \emph{super-additive} and respectively \emph{sub-additive pressure functions} given by
\begin{equation}\label{pressure-function}
P_{\text{sup}}(s):=P_{\text{var}}(f|_{\Lambda},\Psi_f(s)) \text{ and } P_{\text{sub}}(t):=P(f|_{\Lambda},\Phi_f(t)),
\end{equation}
where $P_{\text{var}}(f,\Psi_f(s))$ is the super-additive topological pressure of the potential $\Psi_f(s)$ and $P(f,\Phi_f(t))$ is the sub-additive topological pressure of the potential
$\Phi_f(t)$. It is obvious from the definition of super-and sub-additive topological pressures that $P_{\text{sup}}(s)$ and $P_{\text{sub}}(t)$ are continuous and strictly decreasing in $s$, respectively, in $t$.

\subsection{Lower bound for the Hausdorff dimension of repellers}
For repellers of $C^1$ expanding maps the following result from \cite{bch} provides an upper bound for the Hausdorff dimension of the repeller.
\begin{proposition}\label{lower-dim}
Let $\Lambda$ be a repeller for a $C^1$ expanding map $f$. Then the zero of the sub-additive pressure function $P_{\text{sub}}(t)$ gives an upper bound for the Hausdorff dimension of
$\Lambda$.
\end{proposition}
In view of this result it is natural to ask whether

\emph{Given a $C^1$ expanding map $f$ of an $m_0$-dimensional compact smooth Riemannian manifold $M$ with a repeller $\Lambda$, is it true that $\dim_H\Lambda\ge s^*$, where $s^*$ is the unique root of the equation $P_{\text{sup}}(s)=0$?}

We give here an affirmative answer to this question assuming that $f$ is $C^{1+\gamma}$ for some $\gamma>0$. The reason we need higher regularity of the map is that we will utilize some results from non-uniform hyperbolicity theory.

\begin{theorem}\label{dim-main}
Let $\Lambda$ be a repeller for a $C^{1+\gamma}$ expanding map $f: M\to M$. Then
$$
\dim_H\Lambda\ge s^*,
$$
where $s^*$ is the unique root of the equation $P_{\text{sup}}(s)=0$.
\end{theorem}

As an immediate corollary of this theorem we obtain the following result that gives lower bound for the Hausdorff dimension of the repeller as the zero of the topological pressure of the potential $-\psi^{s}(\cdot, f)$, given by \eqref{sing-val-psi}.
\begin{corollary}\label{cor-dim-main}
Let $\Lambda$ be a repeller for a $C^{1+\gamma}$ expanding map $f: M\to M$. Then
$\dim_H\Lambda\ge s_1$ where $s_1$ is the unique root of the equation
$P(f,-\psi^{s}(\cdot,f))=0$.
\end{corollary}

\subsection{Upper bound for the box dimension of repellers}
Let $\Lambda$ be a repeller for a $C^{1+\gamma}$ expanding map. We assume that $f$ admits a a dominated splitting $T_xM=E_1\oplus E_2\oplus\cdots\oplus E_k$ with $E_1\succeq E_2\succeq\cdots \succeq E_k$ (see Section 5 for the definition). We introduce a sub-additive potential $\widetilde{\Phi}_f(t):=\{-\widetilde{\varphi}^t(\cdot, f^n)\}_{n\ge 1}$ associated with the splitting as follows.


For each $i\in \{1,2,\dots, k\}$ let $\dim E_i=m_i$ and let also
$\ell_d=m_k+\cdots+m_{k-d+1}$ for $d=1,2,\dots, k$ and $\ell_0=0$.  For $t\in [0, m_0]$ and $n\ge 1$, define
\begin{eqnarray*}
\widetilde{\varphi}^t(x, f^n):=
\sum_{j=k-d+1}^{k}m_j\log m(D_xf^n|_{E_j})+(t-\ell_d)\log m(D_xf^n|_{E_{k-d}})
\end{eqnarray*}
if $\ell_d\le   t\le  \ell_{d+1}$ for some $d\in\{0,1,\dots, k-1\}$, where $m(\cdot)$ denotes the minimum norm of an operator. It is easy to see that the potential
$\widetilde{\Phi}_f(t):=\{-\widetilde{\varphi}^t(\cdot, f^n)\}_{n\ge 1}$ is sub-additive and that the corresponding pressure function $\widetilde{P}_{\mathrm{sub}}(t):=P(f|_{\Lambda}, \widetilde{\Phi}_f(t))$ is  continuous and strictly decreasing in $t$.

\begin{theorem}\label{upper-dim}
Let $\Lambda$ be a repeller for a $C^{1+\gamma}$ expanding map admitting a dominated splitting $T_\Lambda M=E_1\oplus E_2\oplus\cdots\oplus E_k$ with
$E_1\succeq E_2\succeq\cdots \succeq E_k$. Then
$$
\overline{\dim}_B\Lambda\le t^*,
$$
where $t^*$ is the unique root of the equation $\widetilde{P}_{\text{sub}}(t)=0$.
\end{theorem}

\subsection{Continuity of sub-additive topological pressure}\label{cont-TP}
Let $f: M\to M$ be a $C^{1+\gamma}$ map of an $m_0$-dimensional compact smooth Riemannian manifold $M$ with metric $d$, which admits a repeller $\Lambda_f$. Let $\|\cdot\|$ be the norm on the tangent space $TM$ that is induced by the Riemannian metric on $M$.

Let $h$ be a $C^1$ map that is sufficiently $C^1$ close to $f$. It is well known that $h$ admits a repeller $\Lambda_h$ such that $(\Lambda_f, f)$ and $(\Lambda_h,h)$ are topologically conjugate that is there is a homeomorphism $\pi:\Lambda_f\to\Lambda_h$ which is close to the identity map and which commutes $f$ and $h$, i.e.,
$\pi\circ f=h\circ \pi$. Recall that $\Phi_h(t)$ is the sub-additive singular valued potential for $h$ for each $0\le t\le m_0$ (see \eqref{sing-val-poten}).

Our first result establishes upper semi-continuity of sub-additive topological pressure of singular valued potential with respect to small $C^1$ perturbations. We stress that our result only requires $f$ to be of class $C^1$.
\begin{theorem}\label{usc-sub-pressure}
Let $f: M\to M$ be a $C^1$ expanding map of an $m_0$-dimensional compact smooth Riemannian manifold $M$, and $\Lambda_f$ a repeller for $f$. Then for each $0\le t\le m_0$ the map
$f\mapsto P(f|_{\Lambda_f},\Phi_f(t))$ is upper semi-continuous, i.e., for every $\varepsilon>0$ there is a $\delta>0$ such that
$$
P(h|_{\Lambda_h},\Phi_h(t))\le P(f|_{\Lambda_f},\Phi_f(t))+\varepsilon
$$
for every $C^1$ map $h:M\to M$ with $\|f-h\|_{C^1}<\delta$.
\end{theorem}
We now establish lower semi-continuity of sub-additive topological pressure of singular valued potential. Note that for this result we require $f$ and its small perturbations to be of class $C^{1+\gamma}$ for some $\gamma>0$.
\begin{theorem}\label{cont-LE1}
Let $f: M\to M$ be a $C^{1+\gamma}$ expanding map of an $m_0$-dimensional compact smooth Riemannian manifold $M$, and $\Lambda_f$ a repeller for $f$. Then for each
$0\le t\le m_0$ the map $f\mapsto P(f|_{\Lambda_f},\Phi_f(t))$ is lower semi-continuous, i.e., for every $\varepsilon>0$ there is a $\delta>0$ such that
$$
P(h|_{\Lambda_h},\Phi_h(t))\ge P(f|_{\Lambda_f},\Phi_f(t))-\varepsilon
$$
for every $C^{1+\gamma}$ map $h:M\to M$ with $\|f-h\|_{C^1}<\delta$.
\end{theorem}
Combining Theorems \ref{usc-sub-pressure} and \ref{cont-LE1}, we have the following result.
\begin{theorem}\label{con-top}
Let $f: M\to M$ be a $C^{1+\gamma}$ expanding map of a $m_0$-dimensional compact smooth Riemannian manifold $M$, and $\Lambda_f$ a repeller for $f$. Then for each
$0\le t\le m_0$ the map $f\mapsto P(f|_{\Lambda_f}, \Phi_f(t))$ is continuous, i.e., for every $\varepsilon>0$ there is a $\delta>0$ such that
$$
P(f|_{\Lambda_f},\Phi_f(t))-\varepsilon\le P(h|_{\Lambda_h},\Phi_h(t))\le P(f|_{\Lambda_f},\Phi_f(t))+\varepsilon
$$
for every $C^{1+\gamma}$ map $h:M\to M$ with $\|f-h\|_{C^1}<\delta$.
\end{theorem}


\section{Carath\'{e}odory singular dimension}\label{cont-cs-dim}

In this section we discuss an approach to the notion of sub-additive topological pressure which is more general than the one presented in Section \ref{sub-sup-pressure} and which is based on the Carath\'{e}odory construction described in \cite{Pes97}. This more general approach allows one to introduce sub-additive topological pressure for arbitrary subsets of the repeller (not necessarily compact or invariant).

Let $\Lambda$ be a repeller for an expanding map $f$ and let $\Phi=\{\varphi_n\}_{n\ge 1}$ be a sub-additive sequence of continuous functions on $\Lambda$. Given a subset $Z\subset \Lambda$ and $s\in\mathbb{R}$, let
\begin{eqnarray}\label{c-p}
m(Z,\Phi, s, r):=\lim_{N\to\infty}\inf\Big\{\sum_i\exp\bigr(-sn_i+\sup_{y\in B_{n_i}(x_i,r)}\varphi_{n_i}(y)\bigr)\Big\},
\end{eqnarray}
where the infimum is taken over all collections $\{B_{n_i}(x_i,r)\}$ of Bowen's balls with $x_i\in \Lambda$, $n_i\ge N$ that cover $Z$. It is easy to show that there is a jump-up value
\[
P_Z(f, \Phi, r)=\inf\{s: m(Z,\varphi,s,r)=0\}=\sup\{s: m(Z,\Phi,s,r)=+\infty\}.
\]
The quantity
\begin{equation}\label{top-pres-car}
P_Z(f,\Phi)=\lim_{r\to 0}P_Z(f,\Phi, r)
\end{equation}
is called the \emph{topological pressure of $\Phi$ on the subset $Z$}, see \cite{fh16} for the weighted version of this quantity.

Since the map $f$ is expanding, the entropy map $\mu\mapsto h_{\mu}(f)$ is upper semi-continuous and the topological entropy is finite. This implies that
 $P_\Lambda(f,\Phi)=P(f|_\Lambda,\Phi)$ i.e., the topological pressure on the whole repeller given by \eqref{top-pres-car} is the same as the topological pressure defined by \eqref{pressure-sub} (see \cite[Proposition 4.4]{cfh} or \cite[Theorem 6.4]{fh16}).

We consider the case when $\Phi$ is the sub-additive singular valued potential, i.e.,
$\Phi=\{-\varphi^{\alpha}(\cdot,f^n)\}_{n\ge 1}$. Let $P(\alpha)=P(f|_\Lambda,\Phi)$ and let
$\alpha_0$ be the zero of Bowen's equation $P(\alpha)=0$. For each sufficiently small $r>0$, it follows from \eqref{c-p} that
\[
m(Z,\Phi, P(\alpha_0), r)=\lim_{N\to\infty}\inf\Big\{\sum_i\exp\bigr(\sup_{y\in B_{n_i}(x_i,r)}-\varphi^{\alpha_0}(y,f^{n_i})\bigr)\Big\},
\]
where the infimum is taken over all collections $\{B_{n_i}(x_i,r)\}$ of Bowen's balls with $x_i\in\Lambda$, $n_i\ge N$ that cover $Z$. This motivates us to introduce the following notion: for every $Z\subset \Lambda$ let
\[
m(Z,\alpha,r):=\lim_{N\to\infty}\inf\Big\{\sum_i\exp\bigr(\sup_{y\in B_{n_i}(x_i,r)}-\varphi^{\alpha}(y,f^{n_i})\bigr)\Big\},
\]
where the infimum is taken over all collections $\{B_{n_i}(x_i,r)\}$ of Bowen's balls with $x_i\in\Lambda$, $n_i\ge N$ that cover $Z$. It is easy to see that there is a jump-up value
\begin{equation}\label{dimC}
\dim_{C,r}(Z):=\inf\{\alpha: m(Z,\alpha)=0\}=\sup\{\alpha: m(Z,\alpha)=+\infty\},
\end{equation}
which we call the \emph{Carath\'{e}odory singular dimension of $Z$}.

We shall show that for $Z=\Lambda$ the Carath\'{e}odory singular dimension of $Z$ is exactly the zero of Bowen's equation $P(\alpha)=0$ and that it varies continuously with $f$.
\begin{theorem}\label{carath}
Let $f:M\to M$ be a $C^{1+\gamma}$ expanding map with repeller
$\Lambda_f$. Assume that $f|\Lambda_f$ is topologically transitive. Then
\begin{enumerate}
\item $\dim_{C,r}\Lambda_f=\alpha_0$ for all sufficiently small $r>0$, where $\alpha_0$ is the unique root of Bowen's equation $P(\alpha)=0$;
\item if $h$ is a $C^{1+\gamma}$ map that is $C^1$ close to $f$, then
$\dim_C(\Lambda_h)$ varies continuously with $h$.
\end{enumerate}
\end{theorem}

\begin{remark} The first statement of Theorem \ref{carath} shows that the Carath\'{e}odory singular dimension of the repeller $\Lambda_f$ is independent of the parameter $r$ for small values of $r>0$. This allows us to use the notation $\dim_C(\Lambda_h)$ for the Carath\'{e}odory singular dimension of the repeller $\Lambda_h$ in the second statement.
\end{remark}


\section{Approximating Lyapunov exponents of expanding maps}\label{ly-exp}

In this section we show that given an ergodic measure $\mu$ on a repeller $\Lambda$ for a $C^{1+\gamma}$ expanding map $f$, its Lyapunov exponents can be well approximated by Lyapunov exponents on a compact invariant subset which carries sufficiently large topological entropy. This statement will serve as a main technical tool in proving our main theorems. We stress that the requirement that the map is of class of smoothness
$C^{1+\gamma}$ is crucial as it allows us to utilize some powerful results of non-uniform hyperbolicity theory.

In what follows $M$ is a $m_0$-dimensional compact smooth Riemannian manifold. We recall the definition of a dominated splitting. Consider a $C^{1+\gamma}$ diffeomorphism of a compact smooth manifold $M$ of dimension $m_0$ and let $\Lambda\subset M$ be a compact invariant set. We say that $\Lambda$ admits a \emph{dominated splitting} if there is continuous invariant splitting $T_\Lambda M=E\oplus F$ and constants $C>0$,
$\lambda\in(0,1)$ such that for each $x\in\Lambda$, $n\in\mathbb{N}$, $0\ne u\in E(x)$, and $0\ne v\in F(x)$
\[
\frac{\|D_xf^n(u)\|}{\|u\|}\le C\lambda^n \frac{\|D_xf^n(v)\|}{\|v\|}.
\]
We write $E\preceq F$ if $F$ dominates $E$. Further, given $0<\ell\le m_0$, we call a continuous invariant splitting $T_\Lambda M=E_1\oplus\cdots\oplus E_\ell$ \emph{dominated} if there are numbers $\lambda_1<\lambda_2<\dots<\lambda_\ell$, constants $C>0$ and $0<\varepsilon<\min_i(\lambda_{i+1}-\lambda_i)/100$ such that for every $x\in \Lambda$, $n\in \mathbb{N}$ and $1\le j\le\ell$ and each unit vector $u\in E_j(x)$
\begin{equation}\label{main-lemma}
C^{-1}\exp(n(\lambda_j-\varepsilon))\le \|D_xf^n(u)\| \le C\exp(n(\lambda_j+\varepsilon)).
\end{equation}
In particular, $E_1\preceq\cdots \preceq E_\ell$. We shall also use the term
\emph{$\{\lambda_j\}$-dominated} when we want to stress dependence of the numbers
$\{\lambda_j\}$.

\begin{theorem}\label{main}
Let $f$ be a $C^{1+\gamma}$ expanding map of $M$ which admits a repeller $\Lambda$, and let $\mu$ be an ergodic measure on $\Lambda$ with $h_{\mu}(f)>0$. Then for any $\varepsilon>0$ there exists an $f$-invariant compact subset $\mathcal{Q}_\varepsilon\subset\Lambda$ such that the following statements hold:
\begin{enumerate}
\item $h_{\text{top}}(f|_{\mathcal{Q}_\varepsilon})\ge h_{\mu}(f)-\varepsilon$;
\item there is a $\{\lambda_j(\mu)\}$-dominated splitting
$T_x M=E_1\oplus E_2\oplus\cdots\oplus E_\ell$ over
$\mathcal{Q}_\varepsilon$ where $\lambda_1(\mu)<\dots<\lambda_\ell(\mu)$ are distinct Lyapunov exponents of $f$ with respect to the measure $\mu$.
\end{enumerate}
\end{theorem}
Consider a $C^{1+\gamma}$ map $h:M\to M$ that is sufficiently $C^1$ close to $f$. Then $h$ is expanding, has a repeller $\Lambda_h$, and the maps $f|\Lambda_f$ and
$h|\Lambda_h$ are topologically conjugate.

We shall show that there is a compact invariant subset $\mathcal{Q}_{\varepsilon}(h)$ of
$\Lambda_h$ of large topological entropy such that the Lyapunov exponents of $\mu$ and of any $h$-invariant ergodic measure $\nu$ with support in $\mathcal{Q}_{\varepsilon}(h)$ are close.

\begin{theorem}\label{cont-LE}Let $(f,\mu)$ be the same as in Theorem \ref{main}. Given a sufficiently small $\varepsilon>0$, there is
$\delta>0$ such that for any $C^{1+\gamma}$ map $h:M\to M$ that is $\delta$-$C^1$ close to $f$ there exists a compact invariant subset
$\mathcal{Q}_{\varepsilon}(h)\subset\Lambda_h$ with the following properties:
\begin{enumerate}
\item $h_{\text{top}}(h|_{\mathcal{Q}_{\varepsilon}(h)})\ge h_{\mu}(f)-\varepsilon$;
\item there is a $\{\lambda_j(\mu)\}$-dominated splitting
$T_xM=E_1\oplus E_2\oplus\cdots\oplus E_\ell$ over
$\mathcal{Q}_\varepsilon(h)$ where $\lambda_1(\mu)<\dots<\lambda_\ell(\mu)$ are distinct Lyapunov exponents of $f$ with respect to the measure $\mu$.
\item for any $h$-invariant ergodic measure $\nu$ with support in
$\mathcal{Q}_{\varepsilon}(h)$ and for each $1\le j\le m_0$, there is $k=k(j)$ such that
$|\lambda_k(\mu)-\lambda_j(\nu)|\le\varepsilon$.
\end{enumerate}
\end{theorem}

We begin with the proof of Theorem \ref{main} and we split it into few steps.

\subsection{Preliminaries on Lyapunov exponents} Let
$\widehat{f}:(\widehat{M},\widehat{\mu})\to(\widehat{M},\widehat{\mu}) $ denote the inverse limit of $f:(M,\mu)\to (M,\mu)$ where
$$
\widehat{M}:=\{\widehat{x}:=\{x_n\}_{n\in\mathbb{Z}}: x_n\in M \text{ and }x_n=f(x_{n+1}) \text{ for all } n\in \mathbb{Z}\}.
$$
We also let $\pi_0:\widehat{M}\to M$ be the projection defined by
$\pi_0(\{x_n\}_{n\in\mathbb{Z}})=x_0$ for any $\{x_n\}_{n\in\mathbb{Z}}\in\widehat{M}$. For each $\widehat{x}=\{x_n\}\in\widehat{M}$ we have that
$\widehat{f}(\widehat{x})=\{f(x_n)\}_{n\in\mathbb{Z}}$, and the inverse of $\widehat{f}$ is given by $\widehat{f}^{-1}(\{x_n\})=\{y_n\}$, where $y_n=x_{n-1}$. We refer the reader to \cite{pu} for more details on inverse limits of non-invertible maps.

Let $\widehat{\Lambda}=\{\widehat{x}=\{x_n\}\in\widehat{M}: x_n\in\Lambda \text{ for all } n\in\mathbb{Z}\}$. It is easy to see that
$\widehat{f}(\widehat{\Lambda})=\widehat{\Lambda}$, and there exists an $\widehat{f}$-invariant ergodic measure $\widehat{\mu}$ supported on
$\widehat{\Lambda}$ such that $(\pi_0)_*\widehat{\mu}=\mu$. It is standard to check that with respect to the cocycle $\{D_{x_0}f^n\}_{n\ge 1}$, for $\widehat{\mu}$-almost every
$\widehat{x}=\{x_n\}\in\widehat{\Lambda}$, the Lyapunov exponents are the same as those
for the system $f:(M,\mu)\to (M,\mu)$. The following result expresses Oseledets' Multiplicative Ergodic Theorem for the system
$\widehat{f}: (\widehat{M},\widehat{\mu})\to (\widehat{M},\widehat{\mu})$.

\begin{proposition}\label{re-p} There exists a full $\widehat{\mu}$-measure Borel set
$\tilde{\Delta}\subset\widehat{\Lambda}$ such that for every
$\widehat{x}=\{x_n\}\in\tilde{\Delta}$ there is an invariant splitting of the tangent space $T_{x_0}M$
\begin{equation}\label{inv-split1}
T_{x_0}M=E_1(\widehat{x})\oplus E_2(\widehat{x})\oplus\cdots\oplus E_{p(\widehat{x})}(\widehat{x})
\end{equation}
and numbers $0<\lambda_1(\widehat{x})<\lambda_2(\widehat{x})<\cdots<  \lambda_{p(\widehat{x})}(\widehat{x})<+ \infty$ and
$m_i(\widehat{x})\,(i=1,2,\dots, p(\widehat{x}))$ such that
$\dim E_i(\widehat{x})=m_i(\widehat{x})$ with
$\sum_{i=1}^{p(\widehat{x})}m_i(\widehat{x})=m_0$ and
$$
\lim\limits_{n\to\pm\infty}\frac1n\log\|T_0^n(\widehat{x})v\|=\lambda_i(\widehat{x})
$$
for each $i$ and for all $0\ne v\in E_i(\widehat{x})$, where
\begin{eqnarray*}
T_0^n(\widehat{x})=\left\{
\begin{array}{cc}
D_{x_0}f^n,& {\rm if}  \ n > 0,\\
Id, & {\rm if} \  n=0, \\
(D_{x_n}f^{-n})^{-1},  &{\rm if} \  n<0.
\end{array}
\right.
\end{eqnarray*}
Moreover, for any $i\neq j\in\{1,2,\cdots, p(\widehat{x})\}$ the angle between $E_i(\widehat{f}(\widehat{x}))$ and $E_j(\widehat{f}(\widehat{x}))$ satisfies that
\[
\lim_{n\to\pm\infty}\frac1n\log \angle(E_i(\widehat{f}^n(\widehat{x})),E_j(\widehat{f}^n(\widehat{x})) )=0.
\]
\end{proposition}
The splitting \eqref{inv-split1} is called the \emph{Oseledets' splitting}.
Since $\widehat{\mu}$ is ergodic, the numbers $p(\widehat{x})$, $\lambda_i(\widehat{x})$ and $m_i(\widehat{x})$ for all $i$ are in fact constants for $\widehat{\mu}$-almost every point $\widehat{x}$. Hence, in the rest of Section 5 we will denote them simply by $p$,
$\lambda_i$ and $m_i$ respectively.

For each $\widehat{x}\in\tilde{\Delta}$ there exists a sequence of norms
$\{\|\cdot\|_{\widehat{x},n}\}_{n=-\infty}^{+\infty}$ such that the following statements hold (see \cite{bp13, bp07} for details): for any sufficiently small $\varepsilon>0$,
\begin{enumerate}
\item for each $1\le i\le p$, all $n,l\in\mathbb{Z}$ and all
$v\in E_i(\widehat{f}^n(\widehat{x}))$
$$
e^{l\lambda_i-|l|\varepsilon}\|v\|_{\widehat{x},n}
\le\|T_n^{l}(\widehat{x})v \|_{\widehat{x}, n+\ell}\le e^{l\lambda_i+|\ell|\varepsilon}
\|v\|_{\widehat{x},n},
$$
where $T_n^{l}(\widehat{x}):=T_0^{l}(\widehat{f}^n(\widehat{x}))$ is as in Proposition
\ref{re-p};
\item for all $v\in T_{x_n}M$
$$
\frac{1}{\sqrt{m_0}}\|v\|\le\|v\|_{\widehat{x},n}\le A(\widehat{x},\varepsilon)
e^{\varepsilon\frac{|n|}{2}}\|v\|,
$$
where $A(\widehat{x},\varepsilon)$ satisfies
$$
A(\widehat{f}^n(\widehat{x}),\varepsilon)\le e^{\varepsilon\frac{|n|}{2}}
A(\widehat{x},\varepsilon)
$$
for all $n\in\mathbb{Z}$ and $\widehat{x}\in\tilde{\Delta}$.
\end{enumerate}
In particular, for all $v\in E_i(\widehat{x})$ we have
$$
\begin{aligned}
e^{\lambda_i-\varepsilon}\|v\|_{\widehat{x},0}&\le\|D_{x_0}f(v)\|_{\widehat{x},1}
&\le e^{\lambda_i+\varepsilon}\|v\|_{\widehat{x},0};\\
e^{-\lambda_i-\varepsilon }\|v\|_{\widehat{x},0}&\le\|D_{x_0}f^{-1} (v)\|_{\widehat{x},-1} &\le e^{-\lambda_i+\varepsilon}\|v\|_{\widehat{x},0}.
\end{aligned}
$$
Given $k\ge 1$, consider the regular set
$$
\Lambda_k=\{\widehat{x}\in\tilde{\Delta} :\, A(\widehat{x},\varepsilon)\le e^{\varepsilon k}\}.
$$
The standard results of non-uniform hyperbolicity theory (see \cite{bp13}) show that regular sets are nested
$$
\Lambda_1\subset\Lambda_2\subset\cdots\subset\Lambda_k\subset\Lambda_{k+1} \subset\cdots
$$
and for every $k\ge 1$ one has $\widehat{f}^{\pm 1}( \Lambda_k)\subset\Lambda_{k+1}$ and $\widehat{\mu}(\Lambda_k)\to 1$ as $k\to\infty$. Hence, for every $\delta>0$, there exists $K\in\mathbb{N}$ such that $\widehat{\mu}(\Lambda_k)>1-\delta$ for every $k\ge K$. We would like to point out that $f$ is uniformly hyperbolic on each regular set $\Lambda_k$, and hence, the angle between different $E_i(\widehat{x})$ is uniformly bounded away from zero.

For $\varepsilon>0$, let $R(\varepsilon)$ denote the ball in $\mathbb{R}^{m_0}$ centered at the origin of radius $\varepsilon$. We now recall some properties of Lyapunov charts
$\{\varphi_{\widehat{x}}: \widehat{x}\in \widehat{\Lambda}\}$ for $\widehat{f}$, see \cite{bp13} for the proofs.

\begin{proposition}
There exists an $\widehat{f}$-invariant Borel set
$\widehat{\Lambda}_0\subset \widehat{\Lambda}$ of full $\widehat{\mu}$-measure such that for any $\varepsilon>0$ there exist: 1) a Borel function
$\ell:\widehat{\Lambda}_0\to [1,+\infty)$ satisfying
$\ell(\widehat{f}^{\pm}(\widehat{x}))\le e^{\varepsilon}\ell(\widehat{x})$ for all
$\widehat{x}\in\widehat{\Lambda}_0$; and 2) a collection of embeddings
$\varphi_{\widehat{x}} : R(\ell(\widehat{x})^{-1})\to M$ such that
\begin{enumerate}
\item[(1)] $\varphi_{\widehat{x}}(0)=x_0$ and $\mathbb{R}^{m_i}(\widehat{x}):=(D_0\varphi_{\widehat{x}})^{-1}(E_i(\widehat{x}))\,(i=1,\dots,r)$ are mutually
orthogonal in $\mathbb{R}^{m_0}$;
\item[(2)] For each $n\in\mathbb{Z}$, let
$H_{\widehat{x}}^n=\varphi^{-1}_{\widehat{f}^n(\widehat{x})}\circ\widehat{f}^n\circ \varphi_{\widehat{x}}$ be the connecting map between the chart at $\widehat{x}$ and the chart at $\widehat{f}^n(\widehat{x})$. Then for each $n\in\mathbb{Z}$, $1\le i\le  p$, and
$v\in\mathbb{R}^{m_i}(\widehat{x})$ we have
$$
e^{n \lambda_i-|n|\varepsilon}|v|\le |D_0 H^n_{\widehat{x}}(v)|\le e^{n\lambda_i+|n|\varepsilon}|v|,
$$
where $|\cdot|$ is the usual Euclidean norm on $\mathbb{R}^{m_0}$;
\item[(3)] let $Lip(g)$ denote the Lipschitz constant of the function $g$, then
$$
Lip(H_{\widehat{x}}-D_0H_{\widehat{x}})\le\varepsilon,\,\, Lip(H^{-1}_{\widehat{x}} - D_0H^{-1}_{\widehat{x}})\le\varepsilon
$$
and
$$
Lip(D_0H_{\widehat{x}})\le\ell(\widehat{x}),\,\,Lip(D_0H^{-1}_{\widehat{x}})\le\ell(\widehat{x});
$$
\item[(4)] For every $v, v'\in R(l(\widehat{x})^{-1})$, one has
$$
K^{-1}_0 d(\varphi_{\widehat{x}}(v),\varphi_{\widehat{x}}(v'))\le |v-v'|\le\ell(\widehat{x}) d(\varphi_{\widehat{x}}(v),\varphi_{\widehat{x}}(v'))
$$
for some universal constant $K_0>0$.
\end{enumerate}
\end{proposition}

\subsection{Extending Oseledets' splitting} Given $\widehat{x} \in \Lambda_k$, we can extend the invariant splitting \eqref{inv-split1} and the Lyapunov metric in the tangent space to a neighborhood
$$
B(\pi_0\widehat{f}^{-i}(\widehat{x}), (re^{-2\varepsilon}e^{-\varepsilon (k+i)})^{\frac{1}{\gamma}})
$$
of $\pi_0\widehat{f}^{-i}(\widehat{x})$ for $ i\ge 0$, where the number $r$ is chosen to satisfy
\begin{equation}\label{eq:number-r}
0<r\le\min\Bigl\{\frac{1}{Km_0}, \frac{e^{\lambda_1}(e^\varepsilon-1)}{K\sqrt{m_0}}\Bigr\}
\end{equation}
and $K$ is the H\"older constant of the differentiable operator $Df$. To see this for every $z\in B(\pi_0\widehat{f}^{-i}(\widehat{x}), (re^{-2\varepsilon}e^{-\varepsilon (k+i)})^{\frac{1}{\gamma}})$ and $v\in T_{z}M$ we translate the vector $v$ to a corresponding vector
$\bar{v}\in T_{\pi_0{\widehat{f}^{-i} (\widehat{x})}}M$ along the geodesic connecting $z$ with $\pi_0\widehat{f}^{-i}(\widehat{x})$ and we set $\|v\|_z=\|\bar{v}\|_{\widehat{x},-i}$. We thus obtain a splitting of the tangent space at $z$
$$
T_{z}M=E_1(z)\oplus\cdots\oplus E_{p}(z),
$$
by translating the splitting
$$
T_{\pi_0{\widehat{f}^{-i}(\widehat{x})}}M=E_1(\pi_0{\widehat{f}^{-i} (\widehat{x})})\oplus \cdots\oplus E_p(\pi_0{\widehat{f}^{-i} (\widehat{x})})
$$
at the point $\pi_0\widehat{f}^{-i}(\widehat{x})$ along the geodesic connecting $z$ with this point.

Let $f^{-1}_{x}$ denote the corresponding inverse branch of the map
$f|_{B(x,\delta)}$ where $\delta$ is chosen such that the map $f|_{B(x,\delta)}$ is invertible for each $x\in \Lambda$. Then we have the following results.

\begin{lemma} \label{estimate-nbhd} Let $\widehat{x}\in\Lambda_k$. Assume that for some $i\ge1$ and some $r>0$ satisfying \eqref{eq:number-r}, the following conditions hold:
\begin{enumerate}
\item $y\in B(\pi_0\widehat{f}^{-i}(\widehat{x}),
(re^{-2\varepsilon}e^{-\varepsilon (k+i)})^{\frac{1}{\gamma}})$;
\item $f(y)\in B(\pi_0 \widehat{f}^{-i+1}(\widehat{x}),
(re^{-2\varepsilon}e^{-\varepsilon (k+i-1)})^{\frac{1}{\gamma}})$.
\end{enumerate}
Then for every $v\in E_j(y)\,(j=1,\dots,p)$ we have
$$
\begin{aligned}
e^{\lambda_j-2\varepsilon}\|v\|_{y}&\le\|D_yf(v)\|_{f(y)}&\le e^{\lambda_j+2\varepsilon}\|v\|_{y},\\
e^{-\lambda_j-2\varepsilon}\|v\|_{y}&\le\|D_yf_{\pi_0 \widehat{f}^{-i-1}(\widehat{x})}^{-1}(v)\|_{y^*}&\le e^{-\lambda_j+2\varepsilon}\|v\|_{y},
\end{aligned}
$$
where $y^*=f_{\pi_0 \widehat{f}^{-i-1}(\widehat{x})}^{-1}(y)$.
\end{lemma}
\begin{proof}[Proof of the lemma] Choose small $\varepsilon>0$ such that
$(\kappa-\varepsilon)^{-1}<e^{-\varepsilon/\gamma}$ and
\[
d(f(x),f(y))>(\kappa-\varepsilon)d(x,y)
\]
for every $x,y\in\Lambda$ with $d(x,y)<r^{1/\gamma}$. Thus, if $y\in B(\pi_0\widehat{f}^{-i}(\widehat{x}),
(re^{-2\varepsilon}e^{-\varepsilon (k+i)})^{\frac{1}{\gamma}})$ then
\begin{eqnarray*}
\begin{aligned}
d(f_{\pi_0 \widehat{f}^{-i-1}(\widehat{x})}^{-1}(y),\pi_0 \widehat{f}^{-i-1}(\widehat{x}))
&=d(f_{\pi_0 \widehat{f}^{-i-1}(\widehat{x})}^{-1}(y),f_{\pi_0 \widehat{f}^{-i-1}(\widehat{x})}^{-1}(\pi_0 \widehat{f}^{-i}(\widehat{x})) )\\
&\le e^{-\varepsilon/\gamma}d(y, \pi_0\widehat{f}^{-i}(\widehat{x}) )\\
&\le (re^{-2\varepsilon}e^{-\varepsilon (k+i+1)})^{\frac{1}{\gamma}},
\end{aligned}
\end{eqnarray*}
i.e., $f_{\pi_0\widehat{f}^{-i-1}(\widehat{x})}^{-1}(y)\in B(\pi_0\widehat{f}^{-i-1}(\widehat{x}),
(r e^{-2\varepsilon}e^{-\varepsilon (k+i+1)})^{\frac{1}{\gamma}})$.
For every $v\in E_j(y)$,
\begin{equation}\label{cone}
\begin{aligned}
\|D_{y}f(v)\|_{f(y)}&=\|D_{y} f(v)\|_{\widehat{x} , -i+1} \\
&\le\|D_{\pi_0 \widehat{f}^{-i}(\widehat{x})}f(v)\|_{\widehat{x},-i+1}+\|D_{y} f(v)-D_{\pi_0 \widehat{f}^{-i}(\widehat{x})}f(v)\|_{\widehat{x},-i+1} \\
&\le\|D_{\pi_0 \widehat{f}^{-i}(\widehat{x})}f(v)\|_{\widehat{x},-i+1}+e^{\varepsilon k}e^{\varepsilon i} \|D_yf(v)-D_{\pi_0\widehat{f}^{-i}(\widehat{x})}f(v)\| \\
&\le e^{\lambda_j+\varepsilon}\|v\|_{\widehat{x},-i}+e^{\varepsilon k}e^{\varepsilon i}K [d(\pi_0 \widehat{f}^{-i}(\widehat{x}), y)]^{\gamma}\|v\| \\
&\le  e^{\lambda_j+\varepsilon}\|v\|_{y}+e^{\varepsilon k}e^{\varepsilon i}K re^{-2\varepsilon}e^{-\varepsilon (k+i)}\sqrt{m_0}\|v\|_{y} \\
&=(e^{\lambda_j+\varepsilon}+K re^{-2\varepsilon}\sqrt{m_0})\|v\|_{y}\\
&\le e^{\lambda_j +2\varepsilon}\|v\|_y
\end{aligned}
\end{equation}
where we use \eqref{eq:number-r} in the last inequality.
Similarly, one can show that $\|D_yf(v)\|_{f(y)}\ge e^{\lambda_j-2 \varepsilon }\|v\|_y$. The second estimate can be proven in a similar fashion.
\end{proof}

Let $\widehat{x}\in\Lambda_k$. Assume that Conditions (1), (2) in Lemma \ref{estimate-nbhd} hold for some $i\ge 1$. We define two cones as follows:
$$
V_j(*, \xi)=\Big\{(v_1, v_2): v_1\in E_1\oplus\cdots\oplus E_j, \,v_2\in E_{j+1}\oplus\cdots \oplus E_p,\,\|v_2\|_{*} < \xi \|v_1\|_{*}\Big\}
$$
and
$$
U_j(*,\xi)=\Big\{(w_1, w_2): w_1\in E_1\oplus\cdots\oplus E_j, \,w_2\in E_{j+1}\oplus\cdots \oplus E_p,\,\|w_1\|_{*} < \xi \|w_2\|_{*}\Big\},
$$
where the point $*=\{f_{\pi_0 \widehat{f}^{-i-1}(\widehat{x})}^{-1}(y), y, f(y)\}$ and $\xi>0$ is a small number. These cones have the following properties.
\begin{lemma}\label{cone-inv}Let $\widehat{x}\in\Lambda_k$ and assume that Conditions (1) and (2) in Lemma \ref{estimate-nbhd} hold for some $i\ge1$ and a number $r>0$ satisfying \eqref{eq:number-r} and the following condition
\begin{equation}\label{eq-n-r1}
r<\frac{e^{\lambda_{j+1}-\varepsilon}-e^{\lambda_{j}+\varepsilon}}{4K}.
\end{equation}
Then the cones defined above are invariant in the following sense: there exist $\xi>0$ and $0<\lambda<1$ (both depending on $r$) such that with
$y^*=f_{\pi_0 \widehat{f}^{-i-1}(\widehat{x})}^{-1}(y)$,
$$
D_y f_{\pi_0 \widehat{f}^{-i-1}(\widehat{x})}^{-1} V_j(y,\xi) \subset V_j( y^*, \lambda \xi) \text{ and }D_yf U_j(y,\xi) \subset U_j( f(y), \lambda \xi).
$$
Moreover, for every $v\in U_j(y,\xi)$,
$$
\| D_yf (v)\|_{f(y)} \ge e^{(\lambda_{j+1} - 3\varepsilon)} \|v\|_y
$$
and for every $v \in V_{j}(y, \xi)$,
$$
\| D_y f_{\pi_0 \widehat{f}^{-i-1}(\widehat{x})}^{-1} (v)\|_{y^*} \ge e^{-(\lambda_j + 3\varepsilon)} \|v\|_y.
$$
\end{lemma}
\begin{proof}[Proof of the lemma] Fix a small number $\xi>0$ and let
\[  D_{\pi_0 \widehat{f}^{-i}(\widehat{x})}f= \left(
\begin{array}{cc}
C_{11}     &   0\\
0               &   C_{22}
\end{array}
\right) \text{ and }
D_yf= \left(
\begin{array}{cc}
D_{11}     &   D_{12}\\
D_{21}              &   D_{22}
\end{array}
\right).
\]
Using the same arguments as in the proof of Lemma \ref{estimate-nbhd}, we obtain that $f_{\pi_0\widehat{f}^{-i-1}(\widehat{x})}^{-1}(y)\in B(\pi_0\widehat{f}^{-i-1}(\widehat{x}),
(r e^{-2\varepsilon}e^{-\varepsilon (k+i+1)})^{\frac{1}{\gamma}})$. Then using an argument similar to the one in \eqref{cone}, we have that
\begin{eqnarray*}
\begin{aligned}
\|D_{y} f-D_{\pi_0 \widehat{f}^{-i}(\widehat{x})}f\|_{f(y)}&=\|D_{y} f-D_{\pi_0 \widehat{f}^{-i}(\widehat{x})}f\|_{{\widehat{x},-i+1}}\\
&\le e^{\varepsilon k}e^{\varepsilon i} \|D_yf-D_{\pi_0\widehat{f}^{-i}(\widehat{x})}f\|\\
&\le e^{\varepsilon k}e^{\varepsilon i}K [d(\pi_0 \widehat{f}^{-i}(\widehat{x}), y)]^{\gamma}\\
&\le e^{\varepsilon k}e^{\varepsilon i}K re^{-2\varepsilon}e^{-\varepsilon (k+i)}\\
&\le Kre^{-2\varepsilon}.
\end{aligned}
\end{eqnarray*}
This yields that
$$
\begin{aligned}
\|C_{11}-D_{11}\|_{f(y)}&\le  Kr \text{ and }& \|C_{22} - D_{22}\|_{f(y)}\le Kr, \\
\|D_{12}\|_{f(y)}&\le Kr \text{ and }& \|D_{21}\|_{f(y)} \le Kr.
\end{aligned}
$$
Therefore, for every $v =( v_1,v_2) \in U_j(y, \xi)$, setting
$D_yf(v)=(v_1^{\prime}, v_2^{\prime})$, we obtain that
$$
\begin{aligned}
\|v_1^{\prime}\|_{f(y)}=\|D_{11} (v_1) + D_{12} (v_2)\|_{f(y)}&\le (e^{\lambda_j + \varepsilon}+Kr)\|v_1\|_y + Kr \|v_2\|_y \\
&\le (\xi e^{\lambda_j+ \varepsilon}+(1+\xi) Kr) \|v_2\|_y
\end{aligned}
$$
and
$$
\begin{aligned}
\|v_2^{\prime}\|_{f(y)} = \|D_{21} (v_1) + D_{22} (v_2)\|_{f(y)}&\ge (e^{\lambda_{j+1}- \varepsilon}-Kr) \|v_2\|_y - Kr \|v_1\|_y\\
&\ge (e^{\lambda_{j+1}- \varepsilon} -(1+\xi) Kr) \|v_2\|_y.
\end{aligned}
$$
Thus,
\[
\|v_1^{\prime}\|_{f(y)}\le \frac{ e^{\lambda_j+ \varepsilon}+(1+\frac 1\xi) Kr}{ e^{\lambda_{j+1}- \varepsilon} -(1+\xi) Kr }\xi\|v_2^{\prime}\|_{f(y)}<\frac{ e^{\lambda_j+ \varepsilon}+ (1+\frac 1\xi)Kr}{ e^{\lambda_{j+1}- \varepsilon} - Kr }\xi\|v_2^{\prime}\|_{f(y)}.
\]
Denote $C:=e^{\lambda_{j+1}-\varepsilon}-e^{\lambda_{j}+\varepsilon}$ and choose  $N>0$ such that $\frac KN<\frac{C}{2}$. Let $\xi=Nr$ and
$$
\lambda:=\frac{e^{\lambda_j+\varepsilon}+(1+\frac{1}{\xi})Kr}{ e^{\lambda_{j+1}-\varepsilon} - Kr}=\frac{e^{\lambda_j+\varepsilon}+Kr+\frac{K}{N}}{e^{\lambda_{j+1}-\varepsilon} - Kr}.
$$
Since by \eqref{eq-n-r1}, $2Kr+\frac{K}{N}<C$, we have that $0<\lambda<1$. It follows that
$$
\|v_1^{\prime}\|_{f(y)} < \lambda \xi \|v_2^{\prime}\|_{f(y)}
$$
and hence, $D_yf U_j(y,\xi) \subset U_j( f(y), \lambda \xi)$. Moreover,
$$
\begin{aligned}
\|D_yf(v)\|_{f(y)} &=\|(v_1^{\prime}, v_2^{\prime})\|_{f(y)} := \max \{ \|v_1^{\prime}\|_{f(y)}, \|v_2^{\prime}\|_{f(y)}\}=\|v_2^{\prime}\|_{f(y)}\\
&\ge (e^{\lambda_{j+1} - \varepsilon} -(1+\xi) Kr) \|v_2\|_{y}
\ge e^{\lambda_{j+1} - 3\varepsilon}  \|v\|_y.
\end{aligned}
$$
This proves the first inequality and the second one can be proven in a similar fashion.
\end{proof}

Given  $x\in\Lambda$ and $m\in \mathbb{N}^+$, let $f_x^{-m}$ denote the corresponding inverse branch of the map $f|_{B(f^{m-1}(x),\delta)}\circ \cdots\circ f|_{B(x,\delta)}$, where
$\delta$ is chosen such that the map $f|_{B(x,\delta)}$ is invertible for each $x\in\Lambda$. The following results can be proved using similar arguments as in the proof of Lemmas \ref{estimate-nbhd} and \ref{cone-inv}.
\begin{corollary}\label{3-re}
Assume that $\widehat{f}^m(\widehat{x})\in\Lambda_k$ for some $m\in \mathbb{N}^+$ and  $\pi_0(\widehat{x})=x$. Then for a sufficiently small number $r$,
$y\in B(f^{m}(x),(re^{-2\varepsilon}e^{-\varepsilon k})^{\frac{1}{\gamma}})$ and $v\in E_j(y)$, $v'\in E_j(f^{-m}_{x} (y))$ we have that
$$
\begin{aligned}
e^{-m(\lambda_j + 2 \varepsilon) }\|v\|_y&\le\|D_yf_x^{-m}(v)\|_{y^*}&\le e^{-m(\lambda_j -2 \varepsilon)}\|v\|_{y}\\
e^{m(\lambda_{j}-2\varepsilon)}\|v'\|_{y^*}&\le\|D_{y^*}f^m(v')\|_{y}&\le e^{m(\lambda_{j} + 2\varepsilon)}\|v'\|_{y^*},
\end{aligned}
$$
where
$y^*=f_x^{-m}(y)\in B(x,(re^{-2\varepsilon}e^{-\varepsilon (k+m)})^{\frac{1}{\gamma}})$.
\end{corollary}

\begin{corollary}\label{add-re-c}
Assume that $\widehat{f}^m(\widehat{x})\in\Lambda_k$ for some $m\in \mathbb{N}^+$ and  $\pi_0(\widehat{x})=x$. Then for a sufficiently small number $r$ and any
$y\in B(f^{m}(x),(r e^{-2\varepsilon} e^{-\varepsilon k})^{\frac{1}{\gamma}})$ there are some $\xi>0$ and $0< \lambda<1$ (both depending on $r$) such that
$$
\begin{aligned}
D_yf_x^{-m}V_j(y,\xi)&\subset V_j(y^*,\lambda^m\xi)&\subset\mathrm{Int}V_j(y^*,\xi ), \\
D_{y^*}f^m U_j(y^*, \xi)&\subset U_j(y,\lambda^m\xi)&\subset\mathrm{Int}U_j(y,\xi)
\end{aligned}
$$
and for each $v\in V_j(y, \xi)$, $w\in U_j(f_x^{-m}(y), \xi)$ we have
$$
\|D_yf_x^{-m}(v)\|_{y^*}\ge e^{ -m(\lambda_j + 3 \varepsilon)}\|v\|_y, \quad
\|D_{y^*} f^m (w)\|_y\ge e^{m(\lambda_{j+1} - 3\varepsilon) }\|w \|_{y^*},
$$
where
$y^*=f_x^{-m}(y)\in B(x,(re^{-2\varepsilon}e^{-\varepsilon (k+m)})^{\frac{1}{\gamma}})$ and $\mathrm{Int}(\cdot)$ denotes the interior of a set.
\end{corollary}

\begin{remark}\label{back-norm} In Corollary \ref{add-re-c}, since
$\widehat{f}^m(\widehat{x})\in\Lambda_k$, we have that  for every
$w\in U_j(f_x^{-m}(y), \xi)$,
\[
\begin{aligned}
e^{\varepsilon k} \|D_{y^*} f^m (w)\|&\ge\|D_{y^*} f^m (w)\|_{\widehat{f}^m(\widehat{x}),0}=\|D_{y^*} f^m (w)\|_y\\
&\ge e^{m(\lambda_{j+1} - 3\varepsilon) }\|w \|_{y^*}\ge e^{m(\lambda_{j+1} - 3\varepsilon) }\|w \|\frac{1}{\sqrt{m_0}}.
\end{aligned}
\]
This yields that $\|D_{y^*} f^m (w)\|\ge e^{m(\lambda_{j+1} - 4\varepsilon) }\|w \|$ for all sufficiently large $m$. Hence, in  Corollary \ref{add-re-c}, for each $v\in V_j(y, \xi)$, $w\in U_j(f_x^{-m}(y), \xi)$ one has for all sufficiently large $m$ that
$$
\|D_yf_x^{-m} (v)\|\ge e^{-m(\lambda_j + 4\varepsilon)}  \|v\|,\quad
\|D_{y^*} f^m (w)\|\ge e^{m(\lambda_{j+1} - 4\varepsilon) } \|w \|.
$$
Similarly, in Corollary \ref{3-re}, for every $v\in E_j(y)$, $v'\in E_j(f^{-m}_{x} (y))$  one can further have  that for all sufficiently large $m$
\begin{eqnarray*}
&&e^{-m(\lambda_j + 3 \varepsilon) }\|v\|\le\|D_yf_x^{-m}(v)\|\le e^{-m(\lambda_j -3 \varepsilon)}\|v\|,\\
&&e^{m(\lambda_{j}-3\varepsilon)}\|v'\|\le\|D_{y^*}f^m(v')\|\le e^{m(\lambda_{j} + 3\varepsilon)}\|v'\|.
\end{eqnarray*}
\end{remark}

\subsection{Constructing a compact invariant set with dominated splitting}
The aim of this section is to construct a compact invariant subset of $\Lambda$ on which the topological entropy of $f$ is close to $h_{\mu}(f)$ and the tangent space over this subset admits a dominated splitting with rates given by Lyapunov exponents of $\mu$. To achieve this we will produce a sufficiently large number of points that have distinct orbits of certain length.

By the Brin-Katok entropy formula (see \cite{kat}), if $\mu$ is an $f$-invariant ergodic measure,  for each $\delta\in (0,1)$ we have
\begin{eqnarray*}
h_{\mu}(f)=\lim\limits_{\tilde{\varepsilon}\to 0}\liminf\limits_{n\to\infty}\frac1n\log N(\mu, n , \tilde{\varepsilon},\delta)=\lim\limits_{\tilde{\varepsilon} \to 0} \limsup\limits_{n \to \infty} \frac 1n \log N(\mu,n,\tilde{\varepsilon},\delta),
\end{eqnarray*}
where $N(\mu,n,\tilde{\varepsilon},\delta)$ denotes the minimal number of Bowen's balls
$\{B_n(x,\tilde{\varepsilon})\}$ that are needed to cover a set of measure at least $1-\delta$. Fix a number $\delta\in (0,1)$ and a small $\varepsilon>0$. Then there exists
$\bar{\varepsilon}>0$ such that for every $\tilde{\varepsilon}\le\bar{\varepsilon}$ one can find a number $N_1$ satisfying: given a set $A$ of measure $>1-2\delta$ and a number $n\ge N_1$, any $(n,\tilde{\varepsilon})$-separated set $E\subset A$ of maximal cardinality satisfies
\begin{equation}\label{C1}
\text{\text{Card}}\,E\ge\exp(n(h_{\mu}(f)-\varepsilon)).
\end{equation}
Let us start by choosing some $\varepsilon\in (0, \lambda_1/3)$. There is a regular set
$\Lambda_{K_1}\subset\widehat{\Lambda}_f$ such that
$\widehat{\mu}(\Lambda_{K_1})> 1-\delta$ and for every
$\widehat{x}=\{x_n\}\in\Lambda_{K_1}$,
$y\in B(\pi_0(\widehat{x}), (r e^{-2 \varepsilon} e^{-\varepsilon K_1})^{\frac{1}{\gamma}})$, and $i \ge 1$ we have
\begin{equation}\label{estim}
\|D_yf^{-i}_{\pi_0 (\widehat{x})}\|_{y^*}  \le e^{-i(\lambda_1-3\varepsilon)},
\end{equation}
where $y^*=f^{-i}_{\pi_0 (\widehat{x})}(y)$. Set
\begin{equation}\label{def-rho}
\rho= (r e^{-2\varepsilon} e^{-\varepsilon K_1})^{\frac{1}{\gamma}}.
\end{equation}
Consider a cover of $\pi_0(\Lambda_{K_1})$ by balls
$B(x_1,\rho/4),\dots, B(x_j,\rho/4)$ with centers $x_i\in\pi_0(\Lambda_{K_1}), i=1,\dots, j$.

Let $\mathcal P=\{P_1,\dots, P_\ell\}$, $\ell\ge j$, be a finite measurable partition of $ M$ such that $\mathcal{P}(x_i)\subset B(x_i, \rho/4)$ for $i=1,2,\dots, j$. Here $\mathcal P(x)$ denote the element of the partition $\mathcal P$  that contains $x$. Such a partition does exist, for example, we can take the first $j$ elements of $\mathcal P$ as follows:
\[
P_1=B(x_1,  \rho/4),\dots, P_k=B(x_k,  \rho/4)\setminus \bigcup_{i=1}^{k-1}P_i\text{ for } 2\le k\le j.
\]
We have the following result.
\begin{lemma}
Given numbers $\varepsilon > 0$ and $\delta > 0$ and a positive measure set $\Lambda_{K_1}$ as above, there exist a positive integer $N_2=N_2(\varepsilon, \delta)$ and a compact subset $\Lambda_2^{\prime}\subset\Lambda_{K_1} $ (possibly depending on $\varepsilon$ and $\delta$) such that $\widehat{\mu}(\Lambda_2^{\prime})>1-2\delta$ and for every $\widehat{x} \in \Lambda_2^{\prime}$ and $n \ge N_2$ we have
\begin{eqnarray*}  \pi_0 \widehat{f}^k(\widehat{x}) \in \mathcal P(\pi_0 (\widehat{x}))   \mbox{ and }  \widehat{f}^k(\widehat{x}) \in \Lambda_{K_1}   \mbox{ for some number } k \in [n, n+\varepsilon n).
\end{eqnarray*}
\end{lemma}
\begin{proof}[Proof of the lemma]
The partition $\mathcal{P}=\{P_1,\dots, P_\ell\}$ of $M$ induces a partition
$\widehat{\mathcal{P}}=\{\widehat{P}_1,\dots,\widehat{P}_\ell\}$ of $\widehat{M}$ given by
$\widehat{P}_i=\{\widehat{x}: \pi_0(\widehat{x})\in P_i\}$. Let
\begin{equation}\label{eq:tau}
\tau=\min_{1\le i\le\ell}\{\varepsilon,\widehat{\mu}(\Lambda_{K_1}\cap\widehat{P}_i)/4\}.
\end{equation}
Since $\mu$ is ergodic, Birkhoff's ergodic theorem implies that for each $i=1,\dots,\ell$ and
$\widehat{\mu}$-almost every $\widehat{x}$ we have
\begin{eqnarray*}
\lim\limits_{n\to\infty}\frac1n\text{Card}\Big\{k \in \{0, \dots, n-1\}: \widehat{f}^k(\widehat{x}) \in\Lambda_{K_1}\cap\widehat{P}_i\Big\}=\widehat{\mu}(\Lambda_{K_1}\cap\widehat{P}_i).
\end{eqnarray*}
Using Egorov's theorem and the regularity of the measure $\mu$, we conclude that there exists a compact set $\Lambda_2^{\prime}$ of
$\widehat{\mu}$-measure at least $1-2\delta$ such that the above convergence is uniform on $\Lambda_2^{\prime}$ for each $i=1,\dots,\ell$.

Hence, for the number $\tau>0$ given by \eqref{eq:tau}, there is a number $N_2>0$ such that for every $\widehat{x}\in\Lambda_2^{\prime}$, every $n\ge N_2$, and all $i=1,\dots,\ell$ we have
\begin{eqnarray*}
\Big|\text{Card}\Big\{k\in\{0, \dots, n-1\}: \widehat{f}^k(\widehat{x})\in\Lambda_{K_1}\cap \widehat{P}_i \Big\}-\widehat{\mu}(\Lambda_{K_1}\cap\widehat{P}_i)n\Big|\le\tau^2 n.
\end{eqnarray*}
Assume that $N_2$ is chosen large enough that $\displaystyle{\min_{1\le i\le\ell}(\widehat{\mu}(\Lambda_{K_1}\cap\widehat{P}_i)-3\tau)N_2\varepsilon >1}$. Thus for every $\widehat{x}\in\Lambda_2^{\prime}$, every $i=1,\dots,\ell$, and every $n\ge N_2$ we have
$$
\begin{aligned}
\text{Card}&\Big\{k \in \{n, \dots, n(1+\varepsilon)-1\}: \widehat{f}^k(\widehat{x}) \in \Lambda_{K_1}\cap\widehat{P}_i\Big\}\\
&=\text{Card}\Big\{k\in\{0,\dots, n(1+\varepsilon)-1\}: \widehat{f}^k(\widehat{x}) \in \Lambda_{K_1}\cap\widehat{P}_i\Big\}\\
&-\text{Card}\Big\{k\in\{0, \dots, n-1\}: \widehat{f}^k(\widehat{x}) \in \Lambda_{K_1}\cap\widehat{P}_i\Big\}\\
&\ge\widehat{\mu} (\Lambda_{K_1}\cap\widehat{P}_i)n(1+\varepsilon)-n(1+\varepsilon) \tau^2-\widehat{\mu} (\Lambda_{K_1}\cap\widehat{P}_i)n-n\tau^2\\
&= n\varepsilon (\widehat{\mu} (\Lambda_{K_1} \cap \widehat{P}_i)-\tau^2) -2n\tau^2\\
&\ge n\varepsilon (\widehat{\mu} (\Lambda_{K_1} \cap \widehat{P}_i) -3\tau) > 1.
\end{aligned}
$$
Clearly, $f^k(\pi_0( \widehat{x})) \in P_i$ if $\widehat{f}^k(\widehat{x}) \in \widehat{P}_i$. In particular, this is true for the index $i$ with $P_i=\mathcal{P}(\pi_0( \widehat{x}))$. This concludes the proof of the lemma.
\end{proof}
Choose $n\ge\max\{N_1, N_2\}$ such that
\begin{equation}\label{estimates}
e^{-n(\lambda_1-3\varepsilon)} m_0<\frac18
\end{equation}
(recall that $m_0=\dim M$) and a maximal $(n,\rho)$-separated subset $E\subset\pi_0(\Lambda_2^{\prime})$. It is easy to see that
$\bigcup_{x\in E} B_n(x, \rho)$ has $\mu$-measure at least $1-2\delta$. We partition the set $E$ into sets $F_k$, $n\le k<n(1+\varepsilon))$, defined by
$$
F_k=\Big\{x \in E: \min \{\ell\in [n, (1+\varepsilon)n)\cap \mathbb{N}: f^\ell(x)\in\mathcal P(x) \}=k\Big\}
$$
that is, having the same return time $k$ to their partition element. Let $m$ be the index satisfying
$$
\text{Card}\, F_m=\max_{n\le k<n+\varepsilon n}\text{\text{Card}}\,F_k.
$$
Since
$$\text{Card}\,E=\sum_{n \le k < n+\varepsilon n} \text{Card}\,F_k,
$$
we obtain that $(\varepsilon n)\text{Card}\,F_m\ge\text{Card}\,E$. Observing that
$\varepsilon n<e^{\varepsilon n}$, we find using \eqref{C1} that
$$
\text{Card}\,F_m\ge\frac{\text{Card}\,E}{\varepsilon n}\ge e^{n(h_{\mu}(f)-2\varepsilon)}.
$$
Choose a point $x_i$ such that the corresponding element $\mathcal{P}(x_i)$ has the property that $\text{Card}(F_m\cap\mathcal{P}(x_i))$ is maximal. We then have
\begin{eqnarray*}
\text{Card}(F_m \cap \mathcal P(x_i)) \ge \frac1j \text{Card}\, F_m \ge \frac1j e^{n(h_{\mu}(f) -2\varepsilon)}.
\end{eqnarray*}
Recall that exactly after $m$ iterations each point $x \in F_m \cap \mathcal{P}(x_i)$ returns to $\mathcal{P}(x_i)$, and hence to $B(x_i, \rho/4)$. Recall also that for each such $x$ there is $\widehat{x}\in \Lambda_{K_1}$ with $\pi_0( \widehat{x})=x$ so that
$\widehat{f}^m(\widehat{x})\in\Lambda_{K_1}$. Given $x\in F_m\cap\mathcal{P}(x_i)$, let
$$
U_x = f_x^{-m}(B(x_i, \rho/2)).
$$
Notice that $f^m(x) \in B(x_i,\rho/4) \subset B(f^m(x),\rho/2)$ and by \eqref{def-rho},
$$
B(f^m(x),\rho/2)\subset B(f^m(x),(re^{-2\varepsilon}e^{-\varepsilon K_1})^{\frac{1}{\gamma}}).
$$
Consequently, by \eqref{estim}, for every $z\in B(x_i,\rho/4)$ we have
$$
\|D_zf_x^{-m}\|_{f_x^{-m}(z)} \le e^{-m(\lambda_1 -3 \varepsilon)}.
$$
Note that $\|\cdot\|_{f_x^{-m}(z)}=\|\cdot\|_{\widehat{x},-m}$,  and by the definition of regular set $\Lambda_{K_1}$ we have
$$
\|D_zf_x^{-m}\| \le  e^{-m (\lambda_1 -3 \varepsilon)}  m_0.
$$
Hence, by \eqref{estimates},
$$
\text{diam }U_x =\text{diam }f_x^{-m}(B(x_i, \rho/2))<e^{-m(\lambda_1-3\varepsilon)} m_0\rho <\frac18\rho.
$$
Thus $U_x \subset B(x, \frac18 \rho)$ and $\overline{U}_x \subset B(x_i, \rho/2)$. For every two distinct points $ x, y\in F_m\cap\mathcal{P} (x_i)$, there exists
$U_x=f_x^{-m}(B(x_i,\rho/2))$ and $U_y=f_y^{-m}(B(x_i,\rho/2))$ so that
$$
\overline{U}_x \subset B(x_i, \rho/2),\,\overline{U}_y\subset B(x_i,\rho/2)\,\,\text{and} \,\, \overline{U}_x \cap\overline{U}_y=\emptyset.
$$
Otherwise, suppose that there exists $z\in \overline{U}_x \cap \overline{U}_y$. Then
\[
d_m(x,z)\le \rho/2\,\,\text{and}\,\, d_m(y,z)\le \rho/2,
\]
which contradicts to $d_m(x,y)\ge d_n(x,y)>\rho$, since $x,y\in F_m$ are $(n,\rho)$-separated.

For every $x, y\in F_m\cap\mathcal{P}(x_i)$, note that $f_y^{-m}(U_x)\subset U_y$.  Therefore, we can consider  the following maps
\begin{eqnarray*}
f_y^{-m}(U_x) \,\,\, \stackrel{f^m}{\longrightarrow}  \,\,\,U_x \,\,\stackrel{f^m}{\longrightarrow} \,\,\,B(f^m(x), \rho).
\end{eqnarray*}
For every $z\in U_x$, $i=1,\dots ,p$ we have
$$
D_zf^m U_i(z, \xi)   \subset U_i(f^m(z), \lambda^{m} \xi),\quad
D_zf_y^{-m} V_i(z, \xi) \subset V_i(f_y^{-m}(z), \lambda^{m} \xi).
$$
Given a point $*\in F_m\cap\mathcal P (x_i)$, consider the map
\begin{eqnarray*}
f_*^{-m}:  f_y^{-m}(U_x) \to f_*^{-m} f_y^{-m} (U_x).
\end{eqnarray*}
By Corollary \ref{add-re-c}, we have for every $z\in f_y^{-m}(U_x)$ and every $i=1,\dots, p$,
$$
D_zf^{m}U_i(z, \xi)\subset U_i(f^{m}(z), \lambda^{m} \xi), \quad
D_zf_*^{-m} V_i(z, \xi) \subset V_i(f_*^{-m}(z), \lambda^{m} \xi).
$$
Let $R_{\varepsilon, 0}=\overline{B(x_i,\rho/2)}$ and for all $l\ge 0$,
$$
R_{\varepsilon,l+1}=\bigcup\limits_{x\in F_m\cap\mathcal{P}(x_i)}f_x^{-m}(R_{\varepsilon,l}).
$$
They form a family of nested non-empty compact sets and hence, setting
$$
R_{\varepsilon}=\bigcap\limits_{l\ge 0} R_{\varepsilon,l},
$$
we obtain a non-empty compact set which is $f^m$-invariant and also $f^{-m}_x$-invariant for all $x\in F_m\cap\mathcal{P}(x_i)$ in the sense that $\displaystyle{R_{\varepsilon}=\bigcup_{x\in F_m\cap\mathcal{P}(x_i)}f_x^{-m}(R_{\varepsilon})}$. We also have for every $l\ge 0$, every $z\in R_{\varepsilon,l}$,  $i=1,\dots, p$, and every point $* \in F_m\cap\mathcal{P}(x_i)$,
\begin{eqnarray*}
D_zf^{m} U_i(z, \xi)\subset\mathrm{Int}U_i(f^{ m}(z),\xi),\,\, D_zf_*^{-m}V_i(z, \xi)\subset \mathrm{Int}V_i(f_*^{- m}(z),  \xi)
\end{eqnarray*}
and
\begin{eqnarray*}
\| D_zf^{m}(v)\|\ge e^{m(\lambda_{i+1}-4\varepsilon)}\|v\|, \,\, \|D_zf_*^{-m}(w)\|\ge e^{-m(\lambda_i+4\varepsilon)}\|w\|
\end{eqnarray*}
for every $v\in U_i(z, \xi)$, $w\in V_i(z, \xi)$.
Therefore, $f^m|_{R_{\varepsilon}}$ is uniformly expanding and it is topologically conjugate to the one-sided full shift over an alphabet with $\text{Card}\,F_m\cap\mathcal{P}(x_i)$ symbols. This implies
$h_{\text{top}}(f^m|_{R_{\varepsilon}})=\log (\text{Card}\,F_m\cap\mathcal P(x_i))$.
Let $\mathcal{Q}_{\varepsilon}=R_{\varepsilon}\cup f(R_{\varepsilon})\cup\cdots\cup f^{m-1}(R_{\varepsilon})$. We wish to show that $\mathcal{Q}_{\varepsilon}$ is the desired compact set. Clearly, $\mathcal{Q}_{\varepsilon}$ is $f$-invariant, and we have that
$$
h_{\text{top}}(f|_{\mathcal{Q}_{\varepsilon}})=\frac1m\log (\text{Card} F_m\cap\mathcal P (x_i)).
$$
Using the fact that $(1+\varepsilon)n > m\ge n$, we obtain
$$
h_{\text{top}}(f|_{\mathcal{Q}_{\varepsilon}})\ge \frac1m\log\frac1j+\frac{n}{m}(h_{\mu}(f) -2\varepsilon)\ge h_{\mu}(f) -3\varepsilon.
$$
We now show existence of a $\{\lambda_j\}$-dominated splitting over
$\mathcal{Q}_{\varepsilon}$ satisfying \eqref{main-lemma}. Let
$\ell=\text{Card} \,F_m\cap\mathcal P (x_i)$. For every $z\in R_{\varepsilon}$ there is a unique sequence $(y_1y_2\dots y_n\dots)\in\Sigma_\ell^+$ (here $\Sigma_\ell^+$ is the one sided full shift over $\ell$ symbols) such that
$$
z=\bigcap_{n > 0} (  f_{y_1}^{-m}\circ f_{y_2}^{-m}  \circ \cdots \circ f_{y_n}^{-m}(\overline{B(x_i, \rho/2)})).
$$
Hence, for every $v\in T_zM$,
$$
\|D_zf_{y_1}^{-m} (v)\|_{f_{y_1}^{-m}(z)} \le e^{-m(\lambda_1 - 3\varepsilon)}\|v\|_z.
$$
Also, by Corollary \ref{add-re-c} and Remark \ref{back-norm} we have that
\begin{equation}\label{eq:n1}
D_zf^{ m} U_i(z,\xi)\subset  \mathrm{Int} U_i(f^{m}(z), \xi),\,\, D_zf_{y_1}^{- m} V_i(z, \xi) \subset  \mathrm{Int} V_i(f_{y_1}^{-m}(z), \xi).
\end{equation}
and for every $v\in U_i(z, \xi)$, $w\in V_i(z, \xi)$,
\begin{equation}\label{eq:n2}
\|D_zf^{ m}(v)\|\ge e^{m(\lambda_{i+1} -4\varepsilon)}\|v\|,\,\, \|D_zf_{y_1}^{- m}(z)(w)\|\ge e^{-m(\lambda_i+4\varepsilon)}\|w\|.
\end{equation}
We shall show that \eqref{eq:n1} and \eqref{eq:n2} imply existence of a continuous splitting on $R_{\varepsilon}$,
$$
T_zM = E_1(z) \oplus E_2(z) \cdots \oplus E_p(z),
$$
and for every $v\in E_i(z)$,
$$
e^{m (\lambda_i-3\varepsilon)}\|v\| \le\|D_zf^mv\| \le e^{m (\lambda_i+3\varepsilon)}\|v\|.
$$
To see this note that for every $z\in R_{\varepsilon}$, there is a unique sequence
$(\sigma_1\sigma_2\dots\sigma_n\dots)\in\Sigma_\ell^+$ such that
$$
z=\bigcap_{n > 0}(f_{\sigma_1}^{-m} \circ f_{\sigma_2}^{-m}\circ\cdots\circ  f_{\sigma_n}^{-m} (\overline{B(x_i, \rho/2)})).
$$
Note that $f_{\sigma_1}^{-m}\circ f_{\sigma_2}^{-m}\circ\cdots\circ f_{\sigma_n}^{-m}(z)\in R_{\varepsilon}$ for every $n\ge 1$. Since the set $R_\varepsilon$ is compact and by construction, the cones $U_i(z,\xi)$ and $V_i(z,\xi)$ depend continuously on the point
$z\in R_\varepsilon$, existence of invariant subspaces $E_i(z)$ follows from \eqref{eq:n1} and \eqref{eq:n2}. Indeed, using the standard techniques in hyperbolicity theory (see for example, \cite{K95}), one can show that
$$
\begin{aligned}
E_i(z)=\Bigl(\bigcap_{\ell\ge0}^{\infty}Df^{\ell m}&U_{i-1}(f_{\sigma_{\ell}^{\prime}}^{-m} \circ\cdots\circ f_{\sigma_{1}^{\prime}}^{-m}(z),\xi)\Bigr)\bigcap \\
&\Bigl(\bigcap_{\ell\ge0}^{\infty} D(f_{\sigma_{1}}^{-m}\circ\cdots\circ f_{\sigma_{\ell}}^{-m}) V_{i}(f^{\ell m}(z),\xi)\Bigr),
\end{aligned}
$$
where $\sigma_{i}^{\prime}$ is defined so that $\pi_0\widehat{f}^{-im}(\widehat{z})\in U_{\sigma_{i}^{\prime}}=f_{\sigma_{i}^{\prime}}^{-m}(B(x_i,\rho/2))$ and $\widehat{z}$ is chosen such that $\widehat{z}\in \widehat{R}_\varepsilon\subset \widehat{\Lambda}$ with
$\pi_0(\widehat{z})=z$.

\subsection{Constructing compact invariant sets with dominated splitting for small perturbations} We present a proof of Theorem \ref{cont-LE}. Recall that  $h$ is a $C^{1+\gamma}$ map which is  sufficiently close to $f$ in the $C^1$ topology.
Following the construction of $R_{\varepsilon}$, we may find a subset $R_{\varepsilon}(h)$ of $\Lambda_h$ that we will briefly recall. Start with the same $F_m\cap\mathcal P(x_i)$ as in the previous section. For every $x\in F_m\cap\mathcal P(x_i)$, let
$\mathrm{diam} U^h_x=\mathrm{diam} h_x^{-m} (B(x_i, \rho/2))$. One has that
$$
\text{diam }U^h_x<\frac18\rho.
$$
Moreover, one has that $U^h_x\subset B(x, \frac18 \rho)$ and
$\overline{U^h_x}\subset B(x_i, \rho/2)$. For every two distinct points
$x,y\in F_m\cap\mathcal P(x_i)$ we have that
$\overline{U^h_x}\cap\overline{U^h_y} = \emptyset$ and $h_y^{-m}(U_x)\subset U_y$. Consider the following maps
\begin{eqnarray*}
h_y^{-m}(U_x^h) \,\,\, \stackrel{h^m}{\longrightarrow}  \,\,\,U_x^h \,\,\stackrel{h^m}{\longrightarrow} \,\,\,B(h^m(x), \rho).
\end{eqnarray*}
For $z\in U^h_x$, $i=1,\dots,p$ the cones have the following properties:
$$
D_zh^m U_i(z, \xi)\subset U_i(h^m(z),\lambda^{m}\xi),\,\, D_zh_y^{-m}V_i(z, \xi)\subset V_i(h_y^{-m}(z), \lambda^{m}\xi).
$$
Let $R^h_{\varepsilon, 0}=\overline{B(x_i, \rho/2)}$ and for all $l\ge 0$,
$$
R^h_{\varepsilon,l+1}=\bigcup\limits_{x\in F_m\cap\mathcal{P}(x_i)}h_x^{-m} (R^h_{\varepsilon,l}).
$$
These sets form a family of nested non-empty compact sets, yielding a non-empty compact set
$$
R_{\varepsilon}(h)=\bigcap\limits_{l\ge 0}R_{\varepsilon,l}^h
$$
which is $h^m$-invariant and also $h^{-m}_x$-invariant for all
$x\in F_m\cap\mathcal P(x_i)$. Note that $h^m|_{R_{\varepsilon}(h)}$ and
$f^m|_{R_{\varepsilon}}$ are topologically conjugate, since both of them are topologically conjugate to the full shift over $\ell$ symbols with
$\ell=\text{Card}\,F_m\cap\mathcal P(x_i)$. We also have for every
$z\in R_{\varepsilon}(h)$, there is a unique sequence $(\sigma_1\sigma_2\dots \sigma_n\dots)\in\Sigma_\ell^+$ such that
$$
z=\bigcap_{n > 0}h_{\sigma_1}^{-m}\circ h_{\sigma_2}^{-m}\circ\cdots\circ h_{\sigma_n}^{-m} (\overline{B(x_i, \rho/2)})
$$
and for every $i=1, 2,\dots, p$  one has
\begin{eqnarray}\label{h-cone1}
D_zh^{ m} U_i(z, \xi)\subset \mathrm{Int} U_i(h^{m}(z),  \xi),\,\, D_zh_{\sigma_1}^{-m}V_i(z, \xi) \subset  \mathrm{Int} V_i(h_{\sigma_1}^{-m}(z),  \xi)
\end{eqnarray}
and
\begin{eqnarray}\label{h-cone2}
\|D_zh^{ m}(v)\|\ge e^{m(\lambda_{i+1}-5\varepsilon)}\|v\|,\,\, \|D_zh_{\sigma_1}^{-m}(w)\|\ge e^{-m(\lambda_i+5\varepsilon)}\|w\|
\end{eqnarray}
for every $v\in U_i(z, \xi)$, $w\in V_i(z, \xi)$.
Hence, $h^m|_{R_{\varepsilon}(h)}$ is uniformly expanding, and the set
$$
\mathcal{Q}_{\varepsilon}(h)=R_{\varepsilon}(h)\cup h(R_{\varepsilon}(h))\cup\cdots\cup h^{m-1}(R_{\varepsilon}(h))
$$
is $h$-invariant. In addition, we have that
$$
h_{\text{top}}(h|_{\mathcal{Q}_{\varepsilon}(h)})=h_{\text{top}}(f|_{\mathcal{Q}_{\varepsilon}(f)})\ge h_{\mu}(f)-3\varepsilon
$$
and the first statement of the theorem follows. To prove the second statement, observe that by \eqref{h-cone1} and \eqref{h-cone2}, using the standard cone technique in \cite{K95}, one can show that there exists a continuous splitting on $R_{\varepsilon}(h)$
$$
T_zM = E_1(z)\oplus E_2(z)\cdots\oplus E_p(z)
$$
so that for every $v \in E_i(z)$,
\begin{eqnarray}\label{LE-close}
e^{m (\lambda_i-4\varepsilon)}\|v\|_{z} \le \|D_zh^m v\|_{h^m(z)}\le e^{m (\lambda_i+4\varepsilon)}\|v\|_z.
\end{eqnarray}
This implies the second statement. The last statement of the theorem follows immediately from \eqref{LE-close}.


\section{Proofs of Main Results}\label{main-thms}

\subsection{Proof of Proposition \ref{sup-add-aprox}}
Fix a positive integer $m$. Since $\mathcal{F}_*(\Psi,\mu)=\sup\frac1m\int\psi_m\,d\mu$, for every $\mu\in\mathcal{M}(X,f)$ we have
$$
h_{\mu}(f)+\frac1m\int\psi_m\,d\mu\le h_{\mu}(f)+\mathcal{F}_*(\Psi,\mu)\le P_{\text{var}}(f,\Psi),
$$
The variational principle for the topological pressure of a single continuous potential (see \eqref{var-principle}) yields that
$$
P(f,\frac{\psi_m}{m})\le P_{\text{var}}(f,\Psi ).
$$
It follows that
$$
\limsup_{n\to\infty} P(f,\frac{\psi_n}{n})\le P_{\text{var}}(f,\Psi).
$$
On the other hand, for each $\mu\in\mathcal{M}(X,f)$ we have that
$$
h_\mu(f)+\mathcal{F}_*(\Psi,\mu)=\lim_{n\to\infty}\Big(h_{\mu}(f)+\frac1n\int\psi_n\,d\mu \Big)\le\liminf_{n\to\infty}P(f,\frac{\psi_n}{n}).
$$
This yields that
$$
P_{\text{var}}(f,\Psi)=\sup\Big\{h_{\mu}(f) + \mathcal{F}_*(\Psi,\mu):\; \mu\in
\mathcal{M}(X,f)\Big\}\le\liminf_{n\to\infty}P(f,\frac{\psi_n}{n})
$$
and completes the proof of the first equality.

To prove the second equality choose $\mu\in\mathcal{M}(X, f)$ and note that
$\mu\in\mathcal{M}(X, f^k)$ for every $k\in\mathbb{N}$. By the variational principle for the topological pressure (see \eqref{var-principle}), for every $\mu\in\mathcal{M}(X, f)$,
$$
\frac1k P(f^k,\psi_k)\ge\frac1k\Bigl(h_\mu(f^k)+\int\psi_k \,d\mu\Bigr)=h_\mu(f)+\frac1k
\int\psi_k\,d\mu.
$$
It follows that
$$
\liminf_{k\to\infty}\frac1kP(f^k,\psi_k)\ge h_\mu(f)+\lim\limits_{k\to\infty}\frac1k\int\psi_k\,d\mu=h_\mu(f)+\mathcal{F}_*(\Psi,\mu).
$$
Since $\mu$ is any measure in $\mathcal{M}(X,f)$, we have that
$$
\liminf_{k\to\infty}\frac1k P(f^k,\psi_k)\ge\sup\Big\{h_\mu(f)+\mathcal{F}_*(\Psi,\mu): \mu \in\mathcal{M}(X,f)\Big\}=P_{\text{var}}(f,\Psi).
$$
For any $k\in\mathbb{N}$ and $\mu\in\mathcal{M}(X, f^k)$ the measure
$\nu:=\frac1k\sum_{i=0}^{k-1}f_*^i\mu$ is $f$-invariant and $h_\nu(f)=\frac1k h_\mu(f^k)$. Since $\{\psi_{nk}(x)\}_{n\ge1}$ is super-additive with respect to $f^k$, we have
\[
\lim_{n\to\infty}\frac1n\int\psi_{nk}\,d\mu=\sup_{n\ge 1}\frac1n\int\psi_{nk}\,d\mu
\ge\int\psi_{k}\, d\mu.
\]
For each $0\le i\le k-1$, the super-additivity of $\{\psi_n(x)\}_{n\ge1}$ with respect to $f$ implies that
\begin{eqnarray*}
\int\psi_{nk}(x)\,df_*^i\mu &\ge& \int\psi_{k-i}(x)\, df_*^i\mu
+\int \psi_{(n-1)k}(f^{k-i}x)\,df_*^i\mu+\int\psi_i(f^{nk-i}x)\, df_*^i\mu \\
&\ge&2m+\int\psi_{(n-1)k}(f^{k}x)\,d\mu=2m+\int\psi_{(n-1)k}(x)\,d\mu,
\end{eqnarray*}
where $m=-\max_{0\le i\le k-1}\|\psi_i\|_\infty$. Summing over $i$ from $0$ to $k-1$, we obtain that
\[
k\int\psi_{nk}(x)\,d\nu\ge 2mk+k\int\psi_{(n-1)k}(x)\,d\mu.
\]
Dividing both sides by $nk$ and letting $n\to\infty$, we find that
\[
k\mathcal{F}_*(\Psi, \nu)\ge\int\psi_{k}\, d\mu.
\]
This implies that
\[
P_{\text{var}}(f,\Psi)\ge h_{\nu}(f)+\mathcal{F}_*(\Psi,\nu)\ge\frac1k(h_{\mu}(f^k)+\int\psi_k\, d\mu).
\]
Since $\mu\in\mathcal{M}(X,f^k)$ can be chosen arbitrary, this yields that
\[
P_{\text{var}}(f,\Psi)\ge\frac1kP(f^k,\psi_{k})
\]
and since $k$ can be chosen arbitrarily, we obtain that
\[
P_{\text{var}}(f,\Psi)\ge\limsup_{k\to\infty}\frac1kP(f^k,\psi_{k})
\]
and the second equality follows.

\subsection{Proof of Theorem \ref{dim-main}}
We split the proof of the theorem into two steps.

\subsubsection{Dimension estimates under the dominated splitting assumption} Let
$\Lambda$ be a repeller for a $C^{1+\gamma}$ expanding map $f: M\to M$. In this subsection we obtain a lower bound of the Hausdorff dimension of the repeller assuming that $f|\Lambda$ possesses a dominated splitting.

Assume that the map $f|\Lambda$ possesses a $\{\lambda_j\}$-dominated splitting $T_\Lambda M=E_1\oplus E_2\oplus\cdots\oplus E_k$ with $E_1\succeq E_2\succeq\cdots \succeq E_k$ and $\lambda_1>\lambda_2>\cdots>\lambda_k$. Let $m_j=\dim E_j$, $r_j=m_1+\cdots +m_j$ for $j\in \{1,2, \cdots, k\}$ and $r_0=0$. For each $s\in [0,m_0]$, $n\ge 1$ and $x\in \Lambda$, define
\[
\widetilde{\psi}^s(x, f^n):=\sum_{j=1}^{d}m_j\log \|D_xf^n|_{E_j}\|+(s-r_d)\log\| D_xf^n|_{E_{d+1}}\|
\]
if $r_{d}\le s\le r_{d+1}$ for some $d\in \{0, 1,\cdots, k-1\}$. It is clear $\widetilde{\Psi}_f(s):=\{-\widetilde{\psi}^s(x, f^n)\}_{n\ge 1}$ is super-additive. Let $\widetilde{P}_{\mathrm{sup}}(s):=P_{\mathrm{var}}(f|_{\Lambda}, \widetilde{\Psi}_f(s))$. One can easily see that
$\widetilde{P}_{\mathrm{sup}}(s)$ is continuous and strictly decreasing in $s$.

\begin{lemma}\label{onestep}
Assume that the map $f|\Lambda$ possesses a $\{\lambda_j\}$-dominated splitting $T_\Lambda M=E_1\oplus E_2\oplus\cdots\oplus E_k$  with $E_1\succeq E_2\succeq\cdots \succeq E_k$ and $\lambda_1>\lambda_2>\cdots>\lambda_k$. Then $\dim_H\Lambda\ge s_1$, where $s_1$  is the unique root of Bowen's equation
$$
P(f|_\Lambda, -\widetilde{\psi}^s(\cdot, f))=0.
$$
\end{lemma}
\begin{proof}[Proof of the lemma]Let $\{P_1, P_2\dots, P_k\}$ be a Markov partition of
$\Lambda$. It follows that there is $\delta>0$ such that for each $i=1,2,\dots, k$ the closed $\delta$-neighborhood $\widetilde{P_i}$ of $P_i$ is such that $\widetilde{P_i}\subseteq U$ (here $U$ is an open neighborhood of $\Lambda$ in the definition of the repeller) and
$\widetilde{P_i}\cap\widetilde{P_j}=\emptyset$ whenever $P_i\cap P_j=\emptyset$. Given an admissible sequence ${\mathbf i}=(i_0i_1\dots i_{n-1})$ and the cylinder
$P_{i_0i_1\dots i_{n-1}}$, we denote by $\widetilde{P}_{i_0i_1\dots i_{n-1}}$ the corresponding cylinder. Note that the $\{\lambda_j\}$-dominated splitting can be extended to $U$, since the splitting is continuous on $\Lambda$. Furthermore, note that $x\mapsto E_i(x)$ is H\"{o}lder continuous on $\Lambda$ since the splitting  $T_\Lambda M=E_1\oplus E_2\oplus\cdots\oplus E_k$ is dominated, and the H\"{o}lder continuity of the map $x\mapsto E_i(x)$ can be extended to $U$,  so is the map $x\mapsto \| D_xf|_{E_i}\|$ for every $i=1,2,\cdots, k$.

Since $P(f|_\Lambda, -\widetilde{\psi}^{s_1}(\cdot, f))=0$ and $\widetilde{\psi}^{s_1}(\cdot, f)$ is a H\"{o}lder continuous function on $\Lambda$, there exists a Gibbs measure $\mu$ such that
$$
K^{-1}\exp\Bigl(-\sum_{i=0}^{n-1}\widetilde{\psi}^{s_1}(f^i(x), f)\Bigr)
\le\mu\bigr(P_{i_0i_1\cdots i_{n-1}}\bigr)
\le K\exp\Bigl(-\sum_{i=0}^{n-1}\widetilde{\psi}^{s_1}(f^i(x), f)\Bigr)
$$
for some constant $K >0$ and every $x\in P_{i_0i_1\dots i_{n-1}}$. Since $T_\Lambda M=E_1\oplus E_2\oplus\cdots\oplus E_k$ is a dominated splitting, the angles between different subspaces $E_i$ are uniformly bounded away from zero. Therefore,  $\widetilde{P}_{i_0i_1\dots i_{n-1}}$ contains a rectangle  of sides
\[
\overbrace{a \| D_{\xi_1} f^n|_{E_1}\|^{-1}, \dots  a\| D_{\xi_1} f^n|_{E_1}\|^{-1}}^{m_1} , \dots, \overbrace{ a\| D_{\xi_k} f^n|_{E_k}\|^{-1},\dots,  a\| D_{\xi_k} f^n|_{E_k}\|^{-1} }^{m_k}
\]
where $a>0$ is a constant and  $\xi_i\in \widetilde{P}_{i_0i_1\dots i_{n-1}}$ for each $i=1,2,\cdots, k$.
Since
$f$ is expanding and the map $x\to \| D_x f|_{E_i}\|^{-1}$ is H\"{o}lder continuous, there exists $C_0>0$ so that
\[
\frac 1 C_0\le \frac{\prod_{j=0}^{n-1}\| D_{f^j(x)} f|_{E_i}\|^{-1}}{\prod_{j=0}^{n-1}\| D_{f^j(y)} f|_{E_i}\|^{-1}}\le C_0
\]
for every $x,y\in \widetilde{P}_{i_0i_1\dots i_{n-1}}$. This together with the fact that $\| D_{\xi_i} f^n|_{E_i}\|^{-1}\ge \prod_{j=0}^{n-1}\| D_{f^j(\xi_i)} f|_{E_i}\|^{-1}~(i=1,2,\cdots,k)$
imply that $\widetilde{P}_{i_0i_1\dots i_{n-1}}$ contains a rectangle  of sides
\[
\overbrace{a_1A_{1}(x,n),\dots,a_1A_{1}(x,n) }^{m_1}, \overbrace{a_1A_{2}(x,n),\cdots,a_1A_{2}(x,n) }^{m_2},\dots,\overbrace{a_1A_{k}(x,n),\dots,a_1A_{k}(x,n) }^{m_k}
\]
for some constant $a_1>0$ and $x\in P_{i_0i_1\dots i_{n-1}}$, where $A_{i}(x,n):=\prod_{j=0}^{n-1}\| D_{f^j(x)} f|_{E_i}\|^{-1}$ for each $i=1,2,\cdots, k$.

Without loss of generality, assume that
$r_i\le s_1< r_{i+1}$ for some $i\in \{ 1,2,\cdots, k\}$ and let
$$
\begin{aligned}
\mathcal{Q}=\Big\{{\mathbf i}=(i_0i_1\cdots &i_{n-1}) : a_1A_{i+1}(x, n )
\le r\text{ for all } x \in P_{i_0i_1\cdots i_{n-1}} ; \\
&\mbox{but}  \  a_1 A_{i+1}(y, n-1 ) > r \  \mbox{for some }  y \in P_{i_0i_1 \cdots i_{n-1}} \Big\}.
\end{aligned}
$$
Therefore, for every ${\mathbf i}=(i_0i_1\cdots i_{n-1})\in \mathcal{Q}$ we have
\begin{eqnarray*}
br<a_1A_{i+1}(x, n)\le r\mbox{ for all } x\in P_{i_0i_1\dots i_{n-1}},
\end{eqnarray*}
where $b=C_0^{-1}\min_{x \in \Lambda}\|D_xf\|^{-1}$. Recall that
$A_1(x,n)\le A_2(x,n)\le\cdots\le A_k(x,n)$.

Let  $B$ be a ball of radius $r$ and $\tilde{B}$ a ball of radius $2r$. Put
$$
\mathcal{Q}_1=\{{\mathbf i} \in\mathcal{Q} | P_{\mathbf{i}} \cap B \ne\emptyset\}.
$$
Hence, for $\mathbf{i}\in\mathcal{Q}_1 $ we have $\widetilde{P}_{\mathbf{i}}\cap\tilde{B}$ contains a rectangle of sides
$$
\begin{aligned}
 &\overbrace{a_1A_1(x,n),\dots, a_1A_1(x,n)}^{m_1},\dots,\overbrace{a_1A_i(x,n),\dots, a_1A_i(x,n)}^{m_i}, \\
 &\overbrace{a_1A_{i+1}(x,n),\dots, a_1A_{i+1}(x,n)}^{m_0-r_i}
 \end{aligned}
$$
It follows that
$$
a_1^{m_0} A_{i+1}(x,n)^{m_0-r_i}A_{i}(x,n)^{m_i}\cdots
A_1(x, n)^{m_1} \le\text{vol}_{m_0}(\widetilde{P}_{\mathbf i}\cap\tilde{B}).
$$
Since $$
\begin{aligned}
A_{i+1}(x,n)^{m_0-r_i}&=A_{i+1}(x,n)^{s_1-r_i} A_{i+1}(x,n)^{m_0-s_1}\\
&\ge A_{i+1}(x,n)^{s_1-r_i}(\frac{b}{a_1})^{m_0-s_1}r^{m_0-s_1}
\end{aligned}$$ we have
\[
a_2 r^{m_0-s_1} A_{i+1}(x,n)^{s_1-r_i} A_{i}(x,n)^{m_i}\cdots
A_1(x, n)^{m_1}\le\text{vol}_{m_0}(\widetilde{P}_{\mathbf i}\cap\tilde{B})
\]
for some constant $a_2>0$.
Therefore,
$$
\sum_{{\bf i}\in\mathcal{Q}_1}a_2 r^{m_0-s_1} A_{i+1}(x,n)^{s_1-r_i}A_{i}(x,n)^{m_i}\cdots
A_1(x, n)^{m_1}\le \text{vol}_{m_0}(\tilde{B})\le 2^{m_0}a_3r^{m_0}
$$
for some constants $a_3>0$.
Hence,
\begin{eqnarray*}
\sum_{{\bf i}\in\mathcal{Q}_1}A_{i+1}(x,n)^{s_1-r_i}A_{i}(x,n)^{m_i}\cdots
A_1(x, n)^{m_1}\le a_4 r^{s_1}
\end{eqnarray*}
for some constant $a_4>0$. On the other hand, one has
\begin{eqnarray*}
\mu(B) &\le&\sum_{(i_0i_1\dots i_{n-1})\in\mathcal{Q}_1}\mu(P_{i_0i_1\dots i_{n-1}})\\
&\le& K\sum_{(i_0i_1\dots i_{n-1})\in\mathcal{Q}_1}\exp\Bigl(-\sum_{i=0}^{n-1}\widetilde{\psi}^{s_1}(f^i(x), f)\Bigr) \\
&=& K\sum_{(i_0i_1\dots i_{n-1})\in\mathcal{Q}_1}A_{i+1}(x,n)^{s_1-r_i}A_{i}(x,n)^{m_i}\cdots
A_1(x, n)^{m_1}\\
 &\le& a_5 r^{s_1}
\end{eqnarray*}
for some constant $a_5>0$.
This  implies that $\dim_H\mu\ge s_1$ and hence, $\dim_H\Lambda\ge\dim_H\mu\ge s_1$.
\end{proof}

Observe that an $f^{2^k}$-invariant measure $\mu$ must be $f^{2^{k+1}}$-invariant. This together with the super-additivity of $\{-\widetilde{\psi}^{s}(\cdot, f^n)\}_{n\ge 1}$ yields that for any $f^{2^k}$-invariant measure $\mu$,
\begin{eqnarray*}
\frac{1}{2^{k+1}} P(f^{2^{k+1}},-\widetilde{\psi}^{s} (\cdot, f^{2^{k+1}}))&\ge& \frac{1}{2^{k+1}} (h_{\mu}(f^{2^{k+1}})+2\int-\widetilde{\psi}^{s} (x, f^{2^k})\,d\mu )\\
&=&\frac{1}{2^k}(h_{\mu}(f^{2^k})+\int-\widetilde{\psi}^{s}(x, f^{2^k})d\mu).
\end{eqnarray*}
Hence,
\begin{eqnarray}\label{monotone}
\frac{1}{2^{k+1}} P(f^{2^{k+1}},-\widetilde{\psi}^{s}(\cdot, f^{2^{k+1}}))\ge\frac{1}{2^{k}} P(f^{2^{k}},-\widetilde{\psi}^{s}(\cdot, f^{2^{k}})).
\end{eqnarray}
By Proposition \ref{sup-add-aprox}, we have
\begin{eqnarray}\label{limit-bound}
\widetilde{P}_{\text{sup}}(s)=\lim_{k\to\infty}\frac{1}{2^{k}}P(f^{2^{k}},-\widetilde{\psi}^{s}(\cdot, f^{2^{k}}))=P_{\text{var}}(f|_\Lambda,\{-\widetilde{\psi}^s(\cdot, f^{n})\}).
 \end{eqnarray}
We shall show that under the same requirements as in the above lemma, one can obtain  sharper dimension estimates.

\begin{lemma}\label{result}
Assume that the map $f|\Lambda$ possesses a $\{\lambda_j\}$-dominated splitting $T_\Lambda M=E_1\oplus E_2\oplus\cdots\oplus E_k$ with $E_1\succeq E_2\succeq\cdots \succeq E_k$ and $\lambda_1>\lambda_2>\cdots>\lambda_k$. Then $\dim_H\Lambda\ge s^*$, where $s^*$ is the unique root of Bowen's equation $\widetilde{P}_{\text{sup}}(s)=0$.
\end{lemma}
\begin{proof}[Proof of the lemma]
Note that for each $n\in\mathbb{N}$, the set $\Lambda$ is also a repeller for $f^{2^n}$.
Using Lemma \ref{onestep}, for every $n\in\mathbb{N}$, one has $\dim_H\Lambda\ge s_n$
where $s_n$ is the unique root of the equation
$$
P(f^{2^n}|_\Lambda, -\widetilde{\psi}^{s} (\cdot, f^{2^n}))=0.
$$
By \eqref{monotone}, we have that $s_n\le s_{n+1}$ and hence, there is a limit
$s^*:= \lim_{n \to\infty} s_n$. We have that $\dim_H\Lambda\ge s^*$. It now follows from
\eqref{limit-bound} that $\widetilde{P}_{\text{sup}}(s^*)=0$.
\end{proof}

\begin{remark} In the proof of Lemmas \ref{onestep} and \ref{result}, we only use the fact that the splitting $T_\Lambda=E_1\oplus E_2\oplus\cdots\oplus E_k$ is dominated. We still put $\{\lambda_j\}-$dominated splitting condition in the statement of the above two lemmas, since it is required that there is a $\{\lambda_j\}-$dominated splitting  over the constructed compact invariant set  in the proof of Theorem \ref{dim-main}.
\end{remark}

\subsubsection{Proof of the theorem}\label{main-proof} Let $m\le s^*< m+1$ be the unique root of $P_{\mathrm{sup}}(s)=0$. We first observe that
\begin{eqnarray*}
&&P_{\text{var}}(f|_{\Lambda}, \{-\psi^{s^*}(\cdot, f^n)\})\\
&&=\sup_{\mu\in\mathcal{M}(f|_\Lambda)}\Big\{h_{\mu}(f)-\lim\limits_{n\to\infty}
\frac1n\int\psi^{s^*}(x, f^n)\,d\mu\Big\} \\
&&=\sup_{\mu\in\mathcal{E}(f|_\Lambda)}\Big\{h_\mu(f)-\bigr(\lambda_1(\mu)+\cdots + \lambda_m(\mu)+(s^*-m)\lambda_{m+1}(\mu)\bigr)\Big\},
\end{eqnarray*}
where $\lambda_1(\mu)\ge\cdots\ge\lambda_{m_0}(\mu)$ are the Lyapunov exponents of
$\mu$. It follows that for every $\varepsilon >0$ there exists an egodic measure
$\mu\in\mathcal{E}(f|_\Lambda)$ such that
\begin{eqnarray*}
P_{\text{var}}(f|_{\Lambda}, \{-\psi^{s^*} (\cdot, f^n)\}) - \varepsilon < h_{\mu}(f) - \bigr(\lambda_1(\mu)+ \cdots + \lambda_m(\mu) + (s^*-m) \lambda_{m+1}(\mu)\bigr ).
\end{eqnarray*}
Applying now Theorem \ref{main} to measure $\mu$ we find a compact $f$-invariant set
$\Lambda_{\varepsilon}\subset\Lambda$ such that
\begin{enumerate}
\item $\Lambda_\varepsilon$ admits a $\{\lambda_j\}$-dominated splitting
$T_{\Lambda_{\varepsilon}}M=E_1\oplus E_2\oplus\cdots\oplus E_k$;
\item $h_{\text{top}}(f|_{\Lambda_{\varepsilon}})\ge h_{\mu}(f)-\varepsilon/2$;
\item $e^{\lambda_j(\mu) -\varepsilon}\|v\|_x \le\| D_xf(v)\|_{f(x)}\le e^{\lambda_j(\mu) + \varepsilon}  \|v\|_x $ for every $v \in E_j,  \ j=1, \cdots, k$.
\end{enumerate}
The third property implies that for each $j=1,\dots, k$ and each ergodic measure $\nu$ supported on $\Lambda_\varepsilon$,
\[
\lambda_j(\mu)-\varepsilon\le \lambda_j(\nu)<\lambda_j(\mu)+\varepsilon,\,\, j=1,2,\cdots, m_0.
\]
By Lemma \ref{result}, we have
\begin{eqnarray}\label{fractal-dim}
\dim_H \Lambda_{\varepsilon} \ge s_{\varepsilon},
\end{eqnarray}
where $s_{\varepsilon}$ is the unique root of the equation
$P_{\text{var}}(f|_{\Lambda_{\varepsilon}}, \{-\widetilde{\psi}^s(\cdot, f^n)\})=0$. If $s_{\varepsilon}\ge s^*$, we have
\[
\dim_H\Lambda \ge \dim_H\Lambda_\varepsilon \ge s_{\varepsilon}\ge s^*
\]
and the desired result follows. Otherwise, assume that $r_d\le s^*\le r_{d+1}$ for some
$d\in \{0,1,\cdots, k-1\}$. By the variational principle for the entropy, there exists an ergodic measure $\nu$ on $\Lambda_\varepsilon$ such that
$h_{\text{top}}(f|_{\Lambda_{\varepsilon}})\le h_\nu(f|_{\Lambda_\varepsilon})+\varepsilon/2$ and hence,
\[
\begin{aligned}
&\lim_{n\to\infty}\frac 1 n \int -\widetilde{\psi}^{s^*}(x, f^n)\, d\nu\\
&=\sum_{j=1}^{d}m_j\lim_{n\to\infty}-\frac 1n\int \log \|D_xf^n|_{E_j}\|\,d\nu+ (s^*-r_d)\lim_{n\to\infty}-\frac 1n\int \log \|D_xf^n|_{E_{d+1}}\|\,d\nu\\
&\le \sum_{j=1}^{r_d}-\lambda_j(\nu) +r_d\varepsilon +\sum_{j=r_d+1}^{m} -\lambda_j(\nu) +(m-r_d)\varepsilon -(s^*-m)\lambda_{m+1}(\nu)+(s^*-m)\varepsilon\\
&\le \sum_{j=1}^{m}-\lambda_j(\nu)-(s^*-m)\lambda_{m+1}(\nu)+m_0\varepsilon,
\end{aligned}
\]
where the second inequality follows from (3). Similarly, one has that
\[
\lim_{n\to\infty}\frac 1 n \int -\widetilde{\psi}^{s^*}(x, f^n)\, d\nu\ge \sum_{j=1}^{m}-\lambda_j(\nu)-(s^*-m)\lambda_{m+1}(\nu)-m_0\varepsilon.
\]
It follows that
\begin{eqnarray*}
&&P_{\text{var}}(f|_\Lambda, \{-\psi^{s^*}(\cdot, f^n)\})-\varepsilon\\
 &&< h_\mu(f) - \bigr(\lambda_1(\mu)+ \cdots +\lambda_m(\mu) + (s^* - m) \lambda_{m+1}(\mu) \bigr) \\
&&\le  h_\nu(f|_{\Lambda_\varepsilon})- \bigr(\lambda_1(\nu)+ \cdots +\lambda_m(\nu) + (s^* - m) \lambda_{m+1}(\nu) \bigr)+(s^*+1)\varepsilon\\
&&\le h_\nu(f|_{\Lambda_\varepsilon}) +\lim_{n\to\infty}\frac 1 n \int -\widetilde{\psi}^{s^*}(x, f^n)\, d\nu+2(m_0+1)\varepsilon\\
&&\le P_{\text{var}}(f|_{\Lambda_{\varepsilon}}, \{-\widetilde{\psi}^{s^*} (\cdot, f^n)\}) + 2(m_0+1)\varepsilon.
\end{eqnarray*}
Since
\begin{eqnarray*}
|s^*-s_{\varepsilon} |\log\kappa\le \Big|P_{\text{var}}(f|_{\Lambda_{\varepsilon}}, \{-\widetilde{\psi}^{s_{\varepsilon}}(\cdot, f^n)\})- P_{\text{var}}(f|_{\Lambda_{\varepsilon}}, \{-\widetilde{\psi}^{s^*} (\cdot, f^n)\})\Big|\le |s^*-s_\varepsilon |\log L,
\end{eqnarray*}
where $L=\max_{x\in\Lambda}\|D_xf\|$, we obtain that
\begin{eqnarray*}
|s^*-s_\varepsilon |\log\kappa\le P_{\text{var}}(f|_{\Lambda_{\varepsilon}},
\{-\widetilde{\psi}^{s_\varepsilon}(\cdot, f^n)\}) - P_{\text{var}}(f|_{\Lambda_{\varepsilon}},
\{-\widetilde{\psi}^{s^*}(\cdot, f^n)\})<2(m_0+2)\varepsilon.
\end{eqnarray*}
Hence,
$$
s_\varepsilon\ge s^* - [2(m_0+2)/\log\kappa]\varepsilon.
$$
This together with \eqref{fractal-dim} yields that
$$
\dim_H\Lambda\ge\dim_H\Lambda_{\varepsilon}\ge s^* -[(m_0+2)/\log\kappa]\varepsilon.
$$
Since $\varepsilon$ can be chosen arbitrary small, this implies that
$\dim_H\Lambda\ge s^*$.

\subsection{Proof of Corollary \ref{cor-dim-main}}

By Theorem \ref{dim-main}, one has $\dim_H\Lambda\ge s^*$ where $s^*$is the unique root of the equation $P_{\text{var}}(f|_\Lambda, -\{\psi^s(\cdot, f^n)\})= 0$. For every
$s\in [0,m_0]$ it follows from \eqref{monotone} and \eqref{limit-bound} that
$$
P_{\text{var}}(f|_\Lambda, -\{\psi^s(\cdot, f^n)\})\ge P(f|_\Lambda, -\psi^s(\cdot, f)).
$$
Hence, $s^*\ge s_1$ where $s_1$ is the unique root of the equation
$P(f|_\Lambda, -\psi^{s}(\cdot, f))= 0$. The desired result immediately follows.

\subsection{Proof of Theorem \ref{upper-dim} } As in the proof of Lemma \ref{onestep}, we may assume that the map $x\mapsto m(D_xf|_{E_i})$ is H\"{o}lder continuous on $U$ for each $i=1,2,\dots, k$.

Choose a number $s$ such that $t^*<s\le m_0$ and assume that
$\ell_d \le s\le \ell_{d+1}$ for some $d\in \{0,1,\dots, k-1\}$ (see Section 3.3 for the definition of $\ell_d$). Since $\widetilde{P}_{\mathrm{sub}}(s) <0$, we may find a positive integer $q$ for which
$$
\sum_{{\mathbf i}\in S_q}e^{-\widetilde{\varphi}^s (y_{\mathbf i},f^q)}< 1
$$
for all $y_{\mathbf i}\in P_{\mathbf i}$, where ${\mathbf i} =(i_0i_1\dots i_{q-1})$ is an admissible sequence and $P_{\mathbf i}$ is a cylinder (see \eqref{cyn}). For any $n\ge 1$, let $B_i(x,nq):=\prod_{j=0}^{n-1}m(D_{f^{jq}(x)}f^q|_{E_i})^{-1}$ for $i=1,2,\dots, k$,  it follows that for all $y_{\mathbf i}\in P_{\mathbf i}$,
$$
\sum_{{\mathbf i} \in S_q} B_{k-d}(y_{\mathbf i},q)^{s-\ell_d}B_{k-d+1}(y_{\mathbf i},q)^{m_{k-d+1}}\cdots B_{k-1}(y_{\mathbf i},q)^{m_{k-1}} B_{k}(y_{\mathbf i},q)^{m_{k}}  < 1.
$$
  Given $0<r\le 1$, set
\begin{eqnarray*}
\begin{aligned}
\mathcal Q = \Big\{{\mathbf i }=(i_0i_1\dots& i_{nq-1}): B_{k-d}(x, nq)\le r \text{ for all }x\in P_{i_0i_1\dots i_{nq-1}} \\
&\text{but }r<B_{k-d}(y, (n-1)q)\text{ for some }
y\in P_{i_0i_1\dots i_{(n-1)q-1}}\Big\}.
\end{aligned}
\end{eqnarray*}
Since $x\mapsto B_i(x,q)$ is H\"{o}lder continuous and $f^q$ is expanding, there exists a $C_0>1$ such that for each $n\ge 1$, all $i\in\{1,2,\dots, k\}$, and any
$x,y\in P_{i_0i_1\dots i_{nq-1}}$,
\[
C_0^{-1}\le \frac{B_i(x,nq)}{B_i(y,nq)}\le C_0.
\]
This implies that for all $x\in P_{i_0i_1\dots i_{nq-1}}$,
$$
Cr<B_{k-d}(x, nq)\le r,
$$
where $C=C_{0}^{-1}\min_{x\in \Lambda} B_{k-d}(x,q)$.
For every admissible sequence $(i_0i_1\dots)$, there is a unique integer $n$  such that
$(i_0, \dots, i_{nq-1})\in\mathcal Q$. In particular,
$\Lambda\subset\bigcup_{{\mathbf i}\in\mathcal{Q}}P_{\mathbf i}$. Note that
$$
m(D_xf^{nq}|_{E_1})^{-1}\le m(D_xf^{nq}|_{E_2})^{-1}\le \cdots\le m(D_xf^{nq}|_{E_k})^{-1}
$$
and for all $i\in \{1,2,\cdots, k\}$ and any $x\in \Lambda$,
$$
m(D_xf^{nq}|_{E_i})^{-1}\le B_i(x,nq).
$$
Since the splitting $T_\Lambda M=E_1\oplus E_2\oplus\cdots \oplus E_k$ is dominated, we conclude that the number of balls of radius $r$ required to cover $\Lambda$ is at most the following number times a constant
\begin{eqnarray*}
\begin{aligned}
&\sum_{{\mathbf i}\in\mathcal{Q}}\frac{m(D_{y_{\mathbf i}}f^{nq}|_{E_k})^{-m_k}}{r^{m_k}}\cdots \frac{m(D_{y_{\mathbf i}}f^{nq}|_{E_{k-d+1}})^{-1}}{r^{m_{k-d+1}}}  \\
&\le\sum_{{\mathbf i}\in\mathcal{Q}}\frac{B_k(y_{\mathbf i},nq)^{m_k}}{B_{k-d}(y_{\mathbf i},nq)^{m_k}}\cdots\frac{B_{k-d+1}(y_{\mathbf i},nq)^{m_{k-d+1}}}{B_{k-d}(y_{\mathbf i},nq)^{m_{k-d+1}}}\\
&=\sum_{{\bf i}\in\mathcal{Q}} B_k(y_{\mathbf i},nq)^{m_k} \cdots B_{k-d+1}(y_{\mathbf i},nq)^{m_{k-d+1}}B_{k-d}(y_{\mathbf i},nq)^{-\ell_d}\\
&=\sum_{{\bf i}\in\mathcal{Q}}B_k(y_{\mathbf i},nq)^{m_k} \cdots B_{k-d+1}(y_{\mathbf i},nq)^{m_{k-d+1}}B_{k-d}(y_{\mathbf i},nq)^{s-\ell_d}B_{k-d}(y_{\mathbf i},nq)^{-s}\\
 &\le  C^{-s} r^{-s}.
 \end{aligned}
\end{eqnarray*}
This implies that $\overline{\dim}_B\Lambda\le s$.
Since the number $s$ can be chosen arbitrary, this implies that
$\overline{\dim}_B\Lambda\le t^*$.

\subsection{Proof of Theorem \ref{usc-sub-pressure}}
Denote by $\pi:\Lambda_f\to\Lambda_h$ the homeomorphism that conjugates $f$ and $h$, i.e., $\pi\circ f=h\circ \pi$. Note that $\pi$ is close to the identity.

Fix $n\ge 1$ and $0\le t\le m_0$ and let $g(x)=-\frac1n \varphi^{t}(x, h^n)$. We have that
\[
P(h|_{\Lambda_h}, g)=P(f|_{\Lambda_f}, g\circ\pi)
\]
and that $P(f|_{\Lambda_f}, g\circ\pi)$ is sufficiently close to
$P(f|_{\Lambda_f}, -\frac1n \varphi^{t}(x, f^n))$. This means that the map
$f\mapsto P(f|_{\Lambda_f}, -\frac1n \varphi^{t}(x, f^n))$  is continuous. On the other hand, we have that
\[
P(f|_{\Lambda_f}, \Phi_f(t))=\lim_{n\to\infty}P(f|_{\Lambda_f}, -\frac1n \varphi^{t}(x, f^n))=\inf_{n\ge 1}P(f|_{\Lambda_f}, -\frac1n \varphi^{t}(x, f^n))
\]
where the first equality is proved in \cite[Proposition 2.1]{bch} and the second one in \cite[Lemma 2]{zhang}. This implies the desired upper semi-continuity of the map
$f\mapsto P(f|_{\Lambda_f}, \Phi_f(t))$.

\subsection{Proof of Theorem \ref{cont-LE1}}\label{s6.6} Fix $0\le t\le m_0$. Since the map $\nu\mapsto h_{\nu}(f)+\mathcal{F}_*(\Phi_f(t),\nu)$ is upper semi-continuous, by the variational principle there is an ergodic equilibrium measure $\mu=\mu_t$ for $\Phi_f(t)$, i.e., we have that
\[
h_{\mu}(f)+\mathcal{F}_*(\Phi_f(t),\mu)=P(f|_{\Lambda_f},\Phi_f(t)).
\]
Let $h$ be a $C^{1+\gamma}$ map that is sufficiently close to $f$ in the $C^1$ topology. Given $\varepsilon>0$, consider the compact invariant sets $\mathcal{Q}_{\varepsilon}(h)$ and $\mathcal{Q}_{\varepsilon}(f)$ constructed in Theorem \ref{cont-LE}. Note that
$\mathcal{Q}_{\varepsilon}(h)=\pi(\mathcal{Q}_{\varepsilon}(f))$ where $\pi$ is the conjugacy map between $f$ and $h$. We have that
$$
h_{\text{top}}(h|_{\mathcal{Q}_{\varepsilon}(h)})=h_{\text{top}}(f|_{\mathcal{Q}_{\varepsilon}})\ge h_\mu(f) -\varepsilon.
$$
Moreover, there is a continuous splitting on $\mathcal{Q}_{\varepsilon}(h)$
$$
T_zM=E_1(z)\oplus E_2(z)\cdots\oplus E_r(z),\,\text{ for all }z\in \mathcal{Q}_{\varepsilon}(h)
$$
such that for every $v\in E_i(z)$,
$$
e^{m (\lambda_i-\varepsilon)}\|v\|_z\le\|D_zh^mv\|_{h^m(z)} \le e^{m (\lambda_i+\varepsilon)}\|v\|_z,
$$
where $\lambda_i=\lambda_i(\mu)$ are the Lyapunov exponents of $\mu$. Thus, for every $h$-invariant measure $\nu$ with $\text{supp }\nu\subset\mathcal{Q}_{\varepsilon}(h)$, its Lyapunov exponents on $E_i(z)$ are between
$\lambda_i -\varepsilon$ and $\lambda_i + \varepsilon$. It follows that
\begin{eqnarray*}
\mathcal{F}_*(\Phi_h(t),\nu)&=&-(\lambda_1(\nu)+\lambda_2(\nu)+\cdots
+\lambda_{[t]}(\nu))-(t-[t])\lambda_{[t]+1}(\nu)\\
&\ge& -(\lambda_1+\lambda_2+\cdots+\lambda_{[t]})-(t-[t])\lambda_{[t]+1}-t\varepsilon\\
&=&\mathcal{F}_*(\Phi_f(t),\mu)-t\varepsilon.
\end{eqnarray*}
Now we choose a $h$-invariant ergodic measure $\nu$ on $\mathcal{Q}_{\varepsilon}(h)$ such that $h_{\text{top}}(h|_{\mathcal{Q}_{\varepsilon}(h)})\le h_\nu(h|_{\mathcal{Q}_{\varepsilon}(h)})+\varepsilon$. This yields that
\begin{eqnarray*}
P(h|_{\Lambda_h},\Phi_h(t)) &\ge& P(h|_{\mathcal{Q}_{\varepsilon}(h)},\Phi_h(t))\\
&\ge& h_\nu(h|_{\mathcal{Q}_{\varepsilon}(h)})+ \mathcal{F}_*(\Phi_h(t),\nu)\\
&\ge&h_{\text{top}}(h|_{\mathcal{Q}_{\varepsilon}(h)})
+\mathcal{F}_*(\Phi_f(t),\mu)-(t+1)\varepsilon \\
&\ge&h_{\mu}(f)+\mathcal{F}_*(\Phi_f(t),\mu)-(t+2)\varepsilon\\
&\ge&P(f|_{\Lambda_f},\Phi_f(t)) - (m_0 +2)\varepsilon.
\end{eqnarray*}
This yields the desired result.

\subsection{Proof of Theorem \ref{carath}}
By Theorem \ref{con-top}, the sub-additive topological pressure
$P_{\text{sub}}(t)=P(f|_{\Lambda_f},\Phi_f(t))$ is continuous at $f$ and so is the zero
$\alpha_0=\alpha_0(f)$ of Bowen's equation $P_{\text{sub}}(t)=0$. Hence, the second statement follows from the first one.

Since the function $P_{\text{sub}}(t)$ is strictly decreasing in $t$, for each $t<\alpha_0$ we have that $P_{\text{sub}}(t)>0$. Fix such a number $t$, and take $\delta>0$ so that $P_{\text{sub}}(t)-\delta>0$. Since $\displaystyle{P_{\text{sub}}(t)=\lim_{r\to 0}P(f|_{\Lambda_f},\Phi_f(t),r)}$, there exists $r_0>0$ such that for each $0<r<r_0$ one has
\[
P_{\text{sub}}(t)-\delta<P(f|_{\Lambda_f},\Phi_f(t),r)<P_{\text{sub}}(t)+\delta.
\]
 Fix such a small $r>0$. It follows from the definition of the topological pressure as the Carath\'eodory singular dimension (see \eqref{dimC}) that
\[
m(\Lambda_f,\Phi_f(t), P_{\text{sub}}(t)-\delta, r)=+\infty.
\]
Hence, for each $K>0$, there exists $L\in \mathbb{N}$ so that for any $N\ge L$ we have that
\begin{eqnarray*}
&&e^{-N(P_{\text{sub}}(t)-\delta)}\inf \Big\{ \sum_{i} \exp\bigr(\sup_{y\in B_{n_i}(x_i,r)}-\varphi^{t}(y,f^{n_i})\bigr) \Big\}\\
&&\ge
\inf \Big\{ \sum_{i} \exp\bigr(-(P_{\text{sub}}(t)-\delta)n_i+\sup_{y\in B_{n_i}(x_i,r)}-\varphi^{t}(y,f^{n_i})\bigr) \Big\}\ge K.
\end{eqnarray*}
This yields that
\[
\inf \Big\{ \sum_{i} \exp\bigr(\sup_{y\in B_{n_i}(x_i,r)}-\varphi^{t}(y,f^{n_i})\bigr) \Big\}
\ge e^{N(P_{\text{sub}}(t)-\delta)} K
\]
where the infimum is taken over all collections $\{B_{n_i}(x_i,r)\}$ of Bowen's balls with $x_i\in \Lambda_f$, $n_i\ge N$, which cover $\Lambda_f$. Letting $N\to \infty$, we have that
\begin{eqnarray}\label{left-crt}
m(\Lambda_f,t,r)=+\infty
\end{eqnarray}
for any $t<\alpha_0$.

On the other hand, for each $t>\alpha_0$ we have that $P_{\text{sub}}(t)<0$. Fix such a number $t$, and take $\widetilde{\delta}>0$ so that $P_{\text{sub}}(t)+\widetilde{\delta}<0$. Similarly, there exists $r_1>0$ such that for each $0<r<r_1$ one has
\[
P_{\text{sub}}(t)-\widetilde{\delta}<P(f|_{\Lambda_f},\Phi_f(t),r)<P_{\text{sub}}(t)+\widetilde{\delta}.
\]
Fix such a small $r>0$. We have that
\[
m(\Lambda_f,\Phi_f(t), P_{\text{sub}}(t)+\widetilde{\delta}, r)=0.
\]
Hence, for each $\xi>0$, there exists $\widetilde{L}\in \mathbb{N}$ so that for any $N\ge \widetilde{L}$ we have that
\begin{eqnarray*}
&&e^{-N(P_{\text{sub}}(t)+\widetilde{\delta})}\inf \Big\{ \sum_{i} \exp\bigr(\sup_{y\in B_{n_i}(x_i,r)}-\varphi^{t}(y,f^{n_i})\bigr) \Big\}\\
&&\le
\inf \Big\{ \sum_{i} \exp\bigr(-(P_{\text{sub}}(t)+\widetilde{\delta})n_i+\sup_{y\in B_{n_i}(x_i,r)}-\varphi^{t}(y,f^{n_i})\bigr) \Big\}\le \xi.
\end{eqnarray*}
This implies that
\[
\inf \Big\{ \sum_{i} \exp\bigr(\sup_{y\in B_{n_i}(x_i,r)}-\varphi^{t}(y,f^{n_i})\bigr) \Big\}
\le \xi e^{N(P_{\text{sub}}(t)+\widetilde{\delta})}
\]
where the infimum is taken over all collections $\{B_{n_i}(x_i,r)\}$ of Bowen's balls with $x_i\in \Lambda_f$, $n_i\ge N$, which cover $\Lambda_f$. Letting $N\to \infty$, we have that
\begin{eqnarray}\label{right-crt}
m(\Lambda_f,t,r)=0
\end{eqnarray}
for any $t>\alpha_0$. Combing \eqref{left-crt} and \eqref{right-crt}, we obtain that
\[
\dim_{C,r}(\Lambda_f)=\alpha_0
\]
for any $0<r<\min\{r_0,r_1\}$.
This completes the proof of the theorem.


\section{Application: a comment on Feng and Shmerkin's result} \label{application}

Let $(X,T)$ be a sub-shift of finite type. Here we assume that subshifts of finite type are  defined on a finite alphabet. Recall that a map $A: X \to  \mathbb{R}^{m_0\times m_0}$ induces a matrix cocycle if for $x\in X$ and $n\in\mathbb{N}$ we have that 
$$
A(x,n)=A(T^{n-1}x) \cdots A(x).
$$
We denote by $X^*$ the collection of finite allowable words in $X$, and let $X_n^*$ be the subset of $X^*$ of words of length $n$. A matrix cocycle $A$ on $X$ is said to be locally constant if $A(x)$ only depends on the first coordinate of $x$, that is if 
$x=(x_0,x_1,\dots,x_n,\dots)$, then $A(x)=A(x_0)$.

For $x\in X$, $s\in [0,m_0]$, and $n\in\mathbb{N}$, define a singular valued function 
$\varphi^s(x,n)$ as follows
$$
\varphi^s(x,n)= \alpha_1(A(x,n))\cdots\alpha_m(A(x,n))\alpha_{m+1}(A(x,n))^{s-m},
$$
where $m = [s]$ and $\alpha_1(A(x,n))\ge \alpha_2(A(x,n))\ge \cdots \ge \alpha_{m_0}(A(x,n))$ are the square roots of the eigenvalues of $A(x,n)^*A(x,n)$.  It is well known that the   singular valued function is sub-multiplicative, i.e.,
$$
\varphi^s(x,{n+m}) \le  \varphi^s(x,n)  \times \varphi^s(T^n x,m).
$$
For $A: X \to  \mathbb{R}^{m_0\times m_0}$, define
$$
P(A,s) = \lim\limits_{n \to \infty} \frac1n \log ( \sum\limits_{\mathbf{i} \in X_n^*} \sup_{y \in [\mathbf{i}]}  \varphi^s(A(y,n))  \in [-\infty, \infty).
$$
The limit can be easily seen to exist, due to sub-multiplicativity of the expression in the  parenthesis.
\begin{theorem}\label{fs}
Let  $A: X\to GL(\mathbb{R},m_0) $ be a H\"{o}lder continuous cocycle. Then the map $P(A,s)\to\mathbb{R}$ is continuous map on $A$.
\end{theorem}
Denote by $\mathcal E_T$ the family of $T$-invariant ergodic measures on $X$. To prove the above theorem, we need the following version of Oseledets' Multiplicative Ergodic Theorem, due to Froyland, Lloyd and Quas \cite{flq} which need the map $T$ to be invertible. If $T: X\to X $ is a one-side subshift of finite type, we obtain the same result by considering its inverse limit space and the induced map.

\begin{theorem} Given a measurable map $A : X\to\mathbb{R}^{m_0\times m_0}$ and a measure $\mu\in\mathcal E_T$ such that
$$
\int \log^+ \|A(x)\| d\mu(x) < +\infty,
$$
there exist $\lambda_1(\mu)>\cdots >\lambda_p(\mu)\ge -\infty$, integers $m_1,\dots, m_p$ with $\sum_{i=1}^pm_i =m_0$, and a measurable family of splittings
$$ 
\mathbb{R}^{m_0}=E_1({x})\oplus E_2({x})\oplus\cdots\oplus E_{p}(x),
$$
such that for $\mu$-almost all $x$ the following holds
\begin{enumerate}
\item $A(x)E_i(x)\subset E_i(Tx)$ with equality if $\lambda_i(\mu)> -\infty$;
\item for all $v\in E_i(x)\backslash\{0\}$,
$$
\lim\limits_{n\to\pm\infty}\frac1n\log |A(x,n)v|=\lambda_i(\mu)
$$
with uniform convergence on any compact subset of $E_i(x)\backslash\{0\}$.
\end{enumerate}
\end{theorem}
By the variational principle for sub-additive topological pressure (see \cite{cfh}), we have that
$$
P(A,s)=\sup\Big\{h_{\mu}(T)+\lim\limits_{n\to\infty}\frac1n\int\log\varphi^s(A(y,n))\,d\mu\Big\}.
$$
Since that map 
$\mu\mapsto h_{\mu}(T)+\lim\limits_{n\to\infty}\frac1n\int\log\varphi^s(A(y,n)) \,d\mu$ is upper semi-continuous, there exists $\mu\in\mathcal E_T$ such that
\begin{eqnarray*}
P(A,s) &=& h_{\mu}(T) + \lim\limits_{n \to \infty} \frac1n \int \log \varphi^s(A(y,n))\, d\mu\\
&=& h_{\mu}(T) + \bigr(\lambda_1(\mu) + \cdots + \lambda_{[s]}(\mu) + (s-[s])\lambda_{[s+1]}(\mu)\bigr).
\end{eqnarray*}
On the other hand,  by the results in \cite{bch}, one has
\begin{eqnarray}
P(A,s) &=& \lim\limits_{k \to \infty} P(T, \frac1k \log \varphi^s(A(\cdot,k))) \\
&=& \inf_{ k \ge 1} P(T, \frac1k \log \varphi^s(A(\cdot,k))).
\end{eqnarray}

\begin{proof}[Proof of Theorem \ref{fs}]
We can follow the proof of Theorem \ref{cont-LE1}. Since
$$
P(A,s)=\inf_{k \ge 1}P(T, \frac1k\log\varphi^s(A(\cdot,k))),
$$ 
and for every $k\in\mathbb{N}$, the map
$$ 
P(T, \frac1k\log\varphi^s(A(\cdot,k)))\to\mathbb{R}$$ is continuous at $A$, the map 
$P(A,s)\to\mathbb{R}$ is upper semi-continuous at $A$.

Next we prove that the map is lower semi-continuous. First, there exists 
$\mu\in\mathcal E_T$ such that
\begin{eqnarray*}
P(A,s) = h_{\mu}(T) + \bigr(\lambda_1(\mu) + \cdots + \lambda_{[s]}(\mu) + (s-[s])\lambda_{[s]+1}(\mu)\bigr).
\end{eqnarray*}
Then for every $\varepsilon > 0$, following the proof of Theorem \ref{main}, we find a compact invariant set $K_{\varepsilon}$ such that
\begin{enumerate}
\item $h_{top}(T|_{K_{\varepsilon}})>h_{\mu}(T)-\varepsilon$;
\item there is $m\in\mathbb{N}$ and for each $x\in K_{\varepsilon}$, $1\le i\le p$ a continuous family of invariant cones $U_i(x)$ such that
$$ 
E_1(x)\oplus\cdots\oplus E_i(x)\subset\text{ Int}U_i, \quad A(x)U_i(x)\subset\text{Int}U_i(Tx),
$$
and for all $v\in U_i(x)\backslash\{0\}$, 
$$ 
|A(x,m) v| \ge e^{(m\lambda_i(\mu)-m\varepsilon)}|v|.
$$
\end{enumerate}
Therefore if $B: X\to\mathbb{R}^{d\times d}$ is close to $A$, we will have the following properties:
\begin{enumerate}
\item for each $x\in K_{\varepsilon}$, $1\le i\le p$,
$$  
B(x) U_i(x) \subset\text{Int}U_i(Tx);
$$
\item for all $v\in U_i(x)\backslash\{0\}$, 
$$ 
|B(x,m)v|\ge e^{(m\lambda_i-2m\varepsilon)}|v|.
$$
\end{enumerate}
Hence, there exists $\nu \in \mathcal E(K_{\varepsilon},T) $ such that
\begin{enumerate}
\item $h_{\nu}(T) = h_{top}(T|_{K_{\varepsilon}})$
\item for $i \in [1,d_1]$,  $ \lambda_i(B,\nu) \ge \lambda_1-2\varepsilon$,
\item for $i \in [\sum_{j=1}^k d_j+1, \sum_{j=1}^{k+1}d_j]$,   $ \lambda_i(B,\nu) \ge \lambda_k-2\varepsilon.$
\end{enumerate}
 Thus for every $s \in [0,m_0]$. We have
 \begin{eqnarray*}
 P(B,s) &\ge& P_{K_{\varepsilon}}(B,s)\\
   &>& h_{\mu}(T) - \varepsilon - [ \lambda_1(B,\nu)+\cdots+ \lambda_{[s]}(B,\nu)+(s-[s]) \lambda_{m+1}(B,\nu)]
 \\ &\ge& P(A,s) - (m_0+1) \varepsilon.
 \end{eqnarray*}
 This gives the lower semi-continuity of sub-additive topological pressure.
\end{proof}

\begin{remark}
If $A(x) \in GL(\mathbb{R},m_0)$, then we can construct a compact invariant set $K_{\varepsilon}$ such that
\begin{enumerate}
\item $h_{top}(T|_{K_{\varepsilon}}) > h_{\mu}(T) - \varepsilon,$
\item there is $m\in\mathbb{N}$ and for each $x\in K_{\varepsilon}$ a continuous invariant dominated splitting such that
$$ 
\mathbb{R}^{m_0} = E_1(x) \oplus \cdots  \oplus E_i(x) \oplus  \cdots  \oplus E_p(x),
$$
and for all $v\in E_i(x)\backslash\{0\}$, 
$$ 
e^{(m\lambda_i-m\varepsilon)}|v|\le |A(x,m)v|\le e^{(m\lambda_i+m\varepsilon)}|v|.
$$
\end{enumerate}
Therefore if $B(x) \in GL(\mathbb{R},m_0)$ is near to $A$, we will have the following properties:
\begin{enumerate}
\item  for each $x\in K_{\varepsilon}$, $1\le i\le p$, there exists a continuous invariant dominated splitting such that
$$ 
\mathbb{R}^{m_0}=E_1(B, x)\oplus\cdots\oplus E_i(B, x)\oplus\cdots\oplus E_p(B, x),
$$
\item for all $v\in E_i(B,x)\backslash\{0\}$, 
$$ 
e^{(m\lambda_i-2m\varepsilon)}|v|\le |B(x,m) v|\le e^{(m\lambda_i+2m\varepsilon)}|v|.
$$
\end{enumerate}
\end{remark}

\begin{remark}
If a matrix cocycle $A$ on $X$ is locally constant, then the above theorem gives the result in \cite{fs14}.
\end{remark}


\bibliographystyle{alpha}
\bibliography{bib}

\end{document}